\documentclass[a4paper,11pt]{amsart}
\usepackage{amsmath}
\usepackage{amssymb}
\usepackage{amstext}
\usepackage{amsbsy}
\usepackage{amsopn}
\usepackage{amsthm}
\usepackage{amscd}
\usepackage{amsxtra}
\usepackage{amsfonts}
\usepackage{enumitem}

\usepackage{ifthen}
\usepackage{xspace}
\usepackage{fancyheadings}
\usepackage{layout}
\usepackage{hyperref}

\usepackage{xcolor}

\input xy
\xyoption{matrix}
\xyoption{arrow}
\xyoption{curve}
\newdir{ (}{{}*!/-5pt/@^{(}}

\newboolean{xlabels}
\newcommand{\xlabel}[1]{
                        \label{#1}
                        \ifthenelse{\boolean{xlabels}}
                                   {\marginpar[\hfill{\tiny #1}]{{\tiny #1}}}
                                   {}
                       }

\newcommand{\CC}{\mathbb{C}}
\newcommand{\RR}{\mathbb{R}}

\newcommand{\PP}{\mathbb{P}}

\newcommand{\FF}{\mathbb{F}}

\newcommand{\sC}{{\mathcal C}}
\newcommand{\sD}{{\mathcal D}}

\newcommand{\sH}{{\mathcal H}}

\newcommand{\sK}{{\mathcal K}}

\newcommand{\sO}{{\mathcal O}}

\newcommand{\sV}{{\mathcal V}}
\newcommand{\sW}{{\mathcal W}}

\newboolean{probleme}
\newcommand{\problem}[1]
           {\ifthenelse{\boolean{probleme}}
                       {{\color{red}(PROBLEM: #1)}}
                       {}
           }
\newboolean{zukuenftiges}
\newcommand{\zukunft}[1]
           {\ifthenelse{\boolean{zukuenftiges}}
                       {{\bf(AUSBAUM\"OGLICHKEIT: #1)\bf}}
                       {}
           }
\newboolean{extras}
\newcommand{\extra}[1]
           {\ifthenelse{\boolean{extras}}
                       {{\bf EXTRA #1 EXTRA\bf}}
                       {}
           }

\newboolean{ignore}
\newcommand{\ignore}[1]
           {\ifthenelse{\boolean{ignore}}
                       {{\bf IGNORE #1 IGNORE\bf}}
                       {}
           }

\DeclareMathOperator{\codim}{codim}

\DeclareMathOperator{\rank}{rank}

\DeclareMathOperator{\sing}{sing}
\DeclareMathOperator{\wdeg}{\deg_a}

\theoremstyle{plain}
\begingroup

\newtheorem{thm}{Theorem}

\newtheorem{lem}[thm]{Lemma}

\newtheorem{prop}[thm]{Proposition}

\numberwithin{thm}{subsection}
\endgroup

\newtheorem*{thm*}{Theorem}
\newtheorem*{conj*}{Conjecture}
\newtheorem*{verm*}{Vermutung}

\theoremstyle{definition}
\newtheorem{defn}[thm]{Definition}
\newtheorem{rem}[thm]{Remark}
\newtheorem{example}[thm]{Example}

\numberwithin{equation}{section}

\newcommand{\nosubsections}{\renewcommand{\thethm}{\thesection.\arabic{thm}}
                            \setcounter{thm}{0}
                           }

\newcommand{\cref}[3]{(\ref{#1}, #2 \ref{#3})}

\date{\today}

\usepackage{amssymb,amsmath}
\usepackage{multirow}

\usepackage{url}  

\setboolean{probleme}{false}
\setboolean{xlabels}{false}

\usepackage{graphicx}

\graphicspath{{Images/}} 

\theoremstyle{definition}

\DeclareMathOperator{\dx}{dx}
\DeclareMathOperator{\dy}{dy}
\DeclareMathOperator{\dz}{dz}

\newcommand{\zoladek}{\.Zo\l\c adek\,}

\begin{document}

\title{Darboux Integrability via Singularities of Invariant Curves at Infinity}

\author[von Bothmer]{Hans-Christian von Bothmer}
\address{Hans-Christian von Bothmer, Fachbereich Mathematik der Universit\"at Hamburg\\
Bundesstra\ss e 55\\
20146 Hamburg, Germany}
\email{hans.christian.v.bothmer@uni-hamburg.de}

\dedicatory{To the memory of Wolfgang Ebeling, for coffee, cookies, and singularities.}

\begin{abstract}
We extend the computation of the invariant $\eta(\omega,C,a)$ defined in \cite{newSolutions} to special points on the line at infinity and show that, as in the affine case, its value is determined purely by the geometry of the integral curve $C$. 
By incorporating points at infinity, the invariant $\eta$ yields effective geometric criteria that certify Darboux integrability in cases not covered by affine data alone.
As an application we construct six new codimension-11 components of the degree-3 center variety.
\end{abstract}

\maketitle

\section*{Introduction}
In 1878 Darboux showed that first integrals of planar polynomial vector fields can be obtained by studying their algebraic invariant curves and the associated cofactors \cite{Darboux}. In particular, if the cofactors of several invariant curves satisfy a nontrivial linear relation, then a rational first integral exists and can be written explicitly as a Darboux product. Darboux also observed—and later results made precise—that the existence of sufficiently many invariant curves forces such a dependence automatically. The number of invariant curves required can in fact be reduced considerably by taking into account the multiplicities of the invariant curves and the singularities {\sl of the differential form} \cite{MultiplicityDarboux}.

In \cite{newSolutions} Darboux theory was developed in a different direction by taking into account the singularities {\sl of the integral curves} themselves. This makes the computation entirely independent of the differential form and hence applies uniformly to all differential forms admitting a fixed configuration of integral curves. While \cite{newSolutions} treated only affine singularities, the present note extends the approach to the line at infinity and, in Theorem~\ref{tEtaCriterion}, treats both affine and infinite singularities on an equal footing.

A central application of Darboux theory is to the Poincaré center problem, which seeks criteria ensuring that a singularity of a planar polynomial vector field is a center rather than a focus \cite{Poincare1885}. The existence of a rational first integral forces trajectories to be closed, so Darboux integrability provides a direct mechanism for producing centers \cite{Darboux}. In the quadratic case every center arises this way, yielding a complete classical description \cite{Dulac}. For cubic systems, many centers are explained by Darboux integrability and substantial classifications are known, but a full classification of Darboux-integrable cubic vector fields remains open.

Lacking a complete classification, a complementary approach is to construct families of Darboux-integrable degree-3 centers, with the hope of eventually capturing all of them and guiding a fuller classification through examples. The first systematic list of such examples was compiled by \zoladek in \cite{ZoladekCorrection}. Numerical evidence over finite fields \cite{survey,ste:2011} seems to suggest that we are still far from a complete list: in codimension 11, for instance, we found only four components of the center variety in the literature \cite{torregrosa}, whereas characteristic-$p$ data seem to suggest there are at least 97.

To assess the usefulness of our refinement of Darboux’s method, we exhibit
six new codimension-11 components of the degree-3 center variety and recover
two previously known ones. More precisely, we work with submaximal sextics
$C$ (degree-6 curves with simple singularities of total Milnor number $18$) for
which the line at infinity is bitangent to $C$. These curves enjoy several
favorable numerical coincidences: the expected counts predict
\begin{enumerate}
\item a unique degree-3 polynomial $1$-form $\omega$ (up to scale) having $C$
      as an integral curve;
\item a single zero of $\omega$ outside $C \cup L_\infty$; and
\item an expected codimension $11$ for the family of normalized differential forms obtained by varying $C$ in its equisingular family.
\end{enumerate}
With a careful choice of the singularity types and component structure of 
$C$, our extension of Darboux’s method to infinity guarantees the existence 
of an integrating factor and hence a center at the extra zero of $\omega$. 

Across all our examples, Darboux integrability cannot be certified from affine data alone—singularities at infinity provide the missing constraints. This shows that extending $\eta$ to singularities at 
infinity can substantially expand the range of accessible examples.

The paper is organized into the following sections:
\begin{enumerate}[label=\arabic*.]
\item We fix notation and recall Darboux’s method.
\item We review the \emph{inverse problem} (finding all differential forms having a given curve as an integral curve). The general solution is due to \cite{inverseProblemWalcherSuggestion}. We recall a version that incorporates the singularities of $C$ following \cite{newSolutions}.
\item We recall the definition of our invariant $\eta$ from \cite{newSolutions} and how, at affine points, $\eta$ can be computed purely from the geometry of $C$.
\item We extend the definition of $\eta$ to the line at infinity and show how to compute it there (Proposition~\ref{pEtaInfinity}). The main result of this paper, our extension of Darboux integrability, is stated in Theorem~\ref{tEtaCriterion}.
\item We obtain an upper bound for the number of zeros of $\omega$ outside $C$ in terms of the singularities of $C$.
\item We summarize what is known about the Poincaré center problem in degree~3 over finite fields.
\item We give a blueprint for constructing codimension-11 components from submaximal sextics and use characteristic-$p$ information to select promising baskets of singularities. We then construct six new components and show how two previously known ones are recovered by the same method.
\end{enumerate}

\medskip
\noindent\textit{Acknowledgment.} 
The author would like to acknowledge the use of ChatGPT for assistance with the English language, improving the flow of the exposition, and researching relevant literature. 
The first draft of the Chern class computations in Lemma~\ref{lChern} and Proposition~\ref{pNumZeros} were also produced by ChatGPT. 
All suggestions generated by ChatGPT were carefully checked and verified by the author.

\section{Preliminaries} \label{sPrelim}
\nosubsections

In this article, we describe a plane autonomous system by a differential form $\omega = P\dx + Q\dy$ where $P, Q \in \CC[x,y]$ are polynomials of degree at most $d$.
\begin{defn}
Let $\omega$ be a differential form, and let $F, \mu \in \CC[[x,y]]$ be power series. If
\[
    dF = \mu \omega
\]
then $F$ is called a {\sl first integral} and $\mu$ an {\sl integrating factor} of $\omega$
\end{defn}

For a given $\omega$ it can often be very difficult to decide, whether a first
integral exists. Darboux realized in 1878 that the existence of algebraic
integral curves can help to answer this question:

\begin{defn}
Let $C \in \CC[x,y]$ be a polynomial and $\{ C=0 \}$ the plane algebraic curve defined by $C$. The zero locus
$\{C=0\}$ is called an {\sl algebraic integral curve} of a differential form $\omega$ if and only if
\[
    dC \wedge \omega |_C = 0 \iff dC \wedge \omega = C\cdot K_C
\]
In this situation the $2$-form $K_C$ is called the {\sl cofactor} of $C$. In a slight abuse of notation,
we  also denote the algebraic curve $\{C=0\}$ by the same letter $C$.
\end{defn}

\begin{thm}[Darboux 1878]
Let $\omega$ be a differential form, $C_1,\dots,C_r$ algebraic integral
curves of $\omega$ and $K_1,\dots,K_r$ their cofactors.
\begin{itemize}
\item If $\sum \alpha_i K_i = -d\omega$ for appropriate $\alpha_i \in \CC$ then
$\mu = \prod C_i^{\alpha_i}$ is a rational integrating factor of $\omega$.
\item If $\sum \alpha_i K_i = 0$ for appropriate $\alpha_i \in \CC$ then
$F = \prod C_i^{\alpha_i}$ is a first integral of $\omega$.
\end{itemize}
\end{thm}

\begin{proof}
For the first claim we calculate:
\begin{align*}
    d(\mu \omega)
    &= d\mu \wedge \omega + \mu d \omega \\
    & = \mu \bigl( \frac{d \mu}{\mu} \wedge \omega + d\omega \bigr) \\
    & = \mu ( d\log \mu \wedge \omega + d\omega) \\
    &= \mu \bigl( \sum \alpha_i d \log C_i \wedge \omega + d \omega \bigr) \\
    &= \mu \bigl( \sum \alpha_i \frac{dC_i}{C_i} \wedge \omega + d \omega \bigr) \\
    &= \mu \bigl( \sum \alpha_i K_i+ d \omega \bigr) \\
    &= 0
\end{align*}
consequently there exists an $F$ with $dF = \mu \omega$.

For the second claim we observe that
$$
    dF = \mu \omega \iff dF \wedge \omega = 0.
$$
We now compute
\begin{align*}
    dF \wedge \omega
    &=  F \Bigl(\frac{d F}{F} \wedge \omega \Bigr)\\
    &=  F (d \log F \wedge \omega)\\
    & = F \sum \alpha_i d\log C_i \wedge \omega \\
    & = F \sum \alpha_i K_i\\
    &= 0
\end{align*}
\end{proof}

Sometimes the following trivial observation is useful:

\begin{lem} \label{lUnionOfIntegralCurves}
Let $C, D$ be plane algebraic curves without common components and $\omega$ a differential form.
Then $C$ and $D$ are integral curves of $\omega$ if and only if $C\cup D$ is an integral curve of $\omega$. 

Furthermore if $K_C$, $K_D$ and $K_{C\cup D}$ are the respective cofactors, we have
\[
K_C+K_D = K_{C\cup D}.
\]
\end{lem}

\begin{proof} Locally, outside the intersection
points of $C$ and $D$, the integral curve condition holds for both curves individually if and only it holds for the union. Since the integral curve condition is a closed condition this equivalence must also hold globally. 

For the cofactors we have
\begin{align*}
	CDK_{C\cup D}  
	&=  d(CD) \wedge \omega \\
	&= \bigl((dC)D+CdD\bigr) \wedge \omega \\
	& = CK_CD + CDK_D  \\
	& = CD(K_C+K_D)
\end{align*}
\end{proof}

\section{The inverse problem}
\nosubsections

For a given configuration of integral curves $\{C_1,\dots,C_n\}$ one can consider the so called {\sl inverse problem} of finding all differential forms $\omega$ that have this curve configuration as integral curve. The inverse problem was solved in \cite{inverseProblemWalcherSuggestion}. Here we recall a version of this solution that takes into account the singularities of the curve configuration from \cite{newSolutions}.

\begin{defn}
Let $C_1,\dots,C_r \in \CC[x,y,z]$ be homogeneous polynomials.
Then
\[
     \begin{pmatrix}
    C_{1x} & C_{1y} & C_1 &  &   \\
    \vdots & \vdots & & \ddots & \\
    C_{rx} & C_{ry} & & & C_r
    \end{pmatrix}
\]
is called the {\sl Darboux Matrix} of the configuration $C_1,\dots,C_r$.
\end{defn}

\begin{prop} \label{pDarbouxMatrix}
Let $C_1,\dots,C_r$ be a configuration of plane curves. A differential form $\omega = P\dx + Q\dy$ has integral curves $C_i$ with cofactors $K_i \dx\dy$ if and only if
\[
    M \cdot (Q,-P, -K_1,\dots,-K_r)^t = 0
\]
where $M$ is the Darboux matrix of $C_1,\dots,C_r$.
\end{prop}

\begin{proof}
The definition of integral curve gives
\begin{align*}
     0 &= dC_i \wedge \omega - C_iK_i \dx\dy \\
       &= (C_{ix}\dx + C_{iy}\dy) \wedge (P\dx + Q\dy) - C_iK_i \dx\dy \\
       &= (C_{ix}Q - C_{iy}P-C_iK_i) \dx\dy.
\end{align*}
Writing these equations in matrix form gives the claimed identity.
\end{proof}

\begin{rem}
We can consider $P,Q,C_i$ and $K_i$ either as affine or as homogeneous polynomials in Proposition \ref{pDarbouxMatrix}. Notice that the condition of the proposition is compatible with homogenization and dehomogenization. Observing that
\[
    d\omega = d(P\dx + Q\dy) = (Q_x-P_y) \dx \wedge \dy
\]
we see that also the conditions of Darboux' Theorem are conditions between $P,Q,C_i$ and $K_i$ which make sense for affine and projective polynomials. These conditions are again compatible with homogenization and dehomogenization.

Notice also that we only homogenize the coefficients of $\dx$, $\dy$ and $\dx \wedge \dy$. This is therefore slightly different from pulling back differential forms from $\CC^2$ to $\PP^2$ where one would for example pull back $\dx$ to
\[
    d\left(\frac{x}{z}\right) = \frac{z\dx -x \dz}{z^2}.
\]
\end{rem}

We now restrict to the case of one integral curve. For the inverse problem this is not a restriction, since, by Lemma \ref{lUnionOfIntegralCurves}, a reduced (possibly reducible) curve $C = C_1\dots C_r$ is an integral curve of a differential form if and only if all its irreducible factors are.

\begin{defn}
Let $C \in \CC[x,y,z]$ be a homogeneous polynomial with no multiple factors and $M_C = (C_x,C_y,C)$ its Darboux matrix. Let
\[
    \sV_C(d) := (\ker M_C)_d
\]
denote the space of degree-$d$ differential forms that admit $C$ as an integral curve.

The vector $H_C := (C_y,-C_x,0)^t$ is always in the kernel of $M_C$ and represents the Hamiltonian vector field associated to $C$. Let
\[
    \sV^H_C(d) := (H_C)_d
\]
be the {\sl vector space of trivial differential forms}. Its elements are of the form $FH_C$, where $F \in \CC[x,y,z]$ is homogeneous of degree $d-e+1$.
\end{defn}

We now turn to singularity theory

\begin{defn}
Let $C \in \CC[x,y,z]$ be a homogeneous polynomial with no multiple factors and no components at infinity, and $P$ a point on the curve defined by $C$. We denote the Tjurina number of $C$ in $P$ by $t(P)$. If $P$ is on the line $z=0$, and $i$ is the intersection number of $z=0$ with $C$ we call
\[
    t_z(P) := t(P) + i - 1
\]
the  {\sl modified Tjurina number}.  
\end{defn}

\begin{example} \label{eSimpleSingularities} The so-called {\sl simple singularities} are listed in Table \ref{tSimple}. The simple singularities are all quasi-homogeneous and therefore the Tjuringa number $t$ is equal to the Milnor number $m$. Furthermore,
for simple singularities, the maximum
number of conditions $c$ a polynomial must satisfy in order to have a singularity of this type is also equal to the Tjuringa number.
\end{example}

\begin{table}
\begin{center}
\begin{tabular}{|c|c|c|c}
\hline
name & local equation & $m=t=c$ \\
\hline
node ($A_1$)        & $x^2-y^2$ & 1  \\
cusp  ($A_2$)        & $x^2-y^3$ & 2  \\
tacnode ($A_3$)    & $x^2-y^4$ & 3  \\
$A_n$             & $x^2-y^{n+1}$ & n  \\
triple point ($D_4$) & $x^3-y^3$ & 4  \\
$D_n$            & $y(x^2-y^{n-2})$ & n  \\
$E_6$            & $x^3-y^4$ & 6 \\
$E_7$            & $x(x^2-y^3)$ & 7  \\
$E_8$            & $x^3-y^5$ & 8  \\
\hline
\end{tabular}
\caption{The simple singularities} \label{tSimple}
\end{center}
\end{table}

\newcommand{\expectedDim}{\delta}

\begin{thm}  \label{pExpected}
Let $C \in \CC[x,y,z]$ be a homogeneous polynomial of degree $e$ with no multiple factors and no components at infinity, let $M_C = (C_x,C_y,C)$ be the Darboux matrix, and $X \subset \PP^2$ the scheme defined by the vanishing of $M_C$. Then
\[
    \dim \sV_C(d) \ge {d-e+3 \choose 2} + {d+1 \choose 2} - (e-1)^2 + \deg X =: \expectedDim
\]
where the first summand is $0$ if $e > d+1$.
We call $\expectedDim$ the {\sl expected dimension} of differential forms with integral curve $C$.

Let $C_{\sing}$ be the set of singular points of $C$ outside of the line at infinity, and $C_\infty$ the set of points of $C$ that lie on the line at infinity. If the singularities of $C \cup \{z=0\}$ at the points in $C_\infty$ are all quasi-homogeneous, then
\[
    \deg X = \sum_{P \in C_{\sing}} t(P) + \sum_{P \in C_\infty} t_z(P)
\]
where $t(P)$, $t_z(P)$ denote the Tjurina and modified Tjurina number of the curve $C$ at $P$.
\end{thm}

\begin{proof}
\cite[Prop.~4.4]{newSolutions}.
\end{proof}

\begin{example}[plane sextics] \label{ePlaneSextics}
Let $C$ be plane sextic curve. We want to find degree $3$ differential forms $\omega$ having $C$ as integral curve. Since $6 > 3+1$ the expected number of such differential forms is
\[
    \expectedDim = {3+1 \choose 2} - (6-1)^2 + \deg X =  \deg X -19
\]
So we need at least $\deg X=20$ ensure the existence of $\omega$.

The maximum Milnor number of a plane sextic with simple singularities is $19$. Such sextics are called {\sl maximal}. To
obtain $\deg X = 20$ we need in addition that $C$ has at least one intersection of mulitplicity $2$ with the line at infinity. Unfortunately in this situation all zeros of $\omega$ that lie either on $C$ or on line at infinity, see Example  \ref{eMaximalInfinity}.

If $C$ is a submaximal sextic, i.e. one with total Milnor number $18$ we need two intersections of multiplicity $2$ (or one of multiplicity $3$) to obtain $\delta>0$. In this situation $\omega$ has generically one zero outside of $C$ and the line at infinity.
The examples of this paper are all of this type.

Since the intersection multiplicity of $C$ with the line at infinity is $6$, the contribution to $\deg X$ from infinity is at most $5$. Hence the smallest possible Milnor number of $C$ with $\delta>0$ is $15$. Plane sextics with simple singularities have been classified; see \cite{UrabeSextics,YangSextics}. These classifications contain hundreds of cases, providing a rich source of examples for interesting cubic differential forms.
\end{example}

\section{an analytic invariant} \label{sAnalytic}
\nosubsections

In \cite{newSolutions} we introduced an analytic invariant $\eta$ attached to a triple $(\omega,\{C_1,\dots,C_r\},a)$ where $\omega$ a differential form in the plane, $C_1,\dots,C_r$ are algebraic integral curves of $\omega$ an $a$ is any point in the plane:

\begin{defn} \label{dEta}
Let $\omega$ be a differential form with algebraic integral curves $C_1,\dots,C_r$, and cofactors $K_1,\dots,K_r$ and $a$ a zero of $\omega$. Then we define
\[
    \eta := \eta(\omega,\{C_1,\dots,C_r\},a) := \bigl(K_1(a) : \dots : K_r(a) : d\omega(a) \bigr)
\]
where we interpret the right-hand side as a ratio i.e
\[
(u_1:\dots:u_{n+1}) = (\lambda u_1 :\dots : \lambda u_{n+1}) \quad \text{for $\lambda \not =0.$}
\]
The degenerate case $(0:\dots:0)$ is also allowed.
\end{defn}

While $\eta$ is defined everywhere, it carries the most information if $a$ is a zero of $\omega$.
The main property of $\eta$ is, that it depends in many situations only on the singularities of the curve configuration, but not on $\omega$. To make this precise we need the following:

\begin{defn}
Let $C$ be a algebraic curve and $a \in C$ a point. We say that $C$ has a quasi-homogeneous singularity at $a$ if there exists local analytic coordinates $x_a,y_a$ around $a$ and a weighted grading $\wdeg$ on $\CC[x_a,y_a]$ such that the local equation $C_a$ of $C$ at $a$ is quasi-homogeneous with respect to $\wdeg$.
\end{defn}

\begin{example}
The curve $C = x^2 - y^3$ has a quasi-homogeneous singularity at $a=(0,0)$, since for $x_a=x$ and $y_a=y$ the equation local equation $C_a = C$ is quasi-homogeneous with respect to the grading given by $\wdeg x_a = 3$ and $\wdeg y_a = 2$. Indeed all monomials of $C_a$ have degree $6$ in this grading
\end{example}

\begin{defn}
Let $C = C_1\cup\cdots\cup C_r$ be a plane curve. A point $P \in C$ is called {\sl $\eta$-geometric} if
\begin{enumerate}
\item $C$ has a quasi-homogeneous singularity at $P$.
\item $P$ is {\sl not} a point where exactly two distinct curves of $\{C_1,\dots,C_r\}$ intersect transversally.
\end{enumerate}
\end{defn}

\begin{thm}\label{tEtaQuasiHom}
Let $C=C_1\cup\cdots\cup C_r$ be an algebraic integral curve for a differential form $\omega$, and let $a$ be an $\eta$–geometric point of $C$ not lying on the line at infinity. Then either
\[
  \eta=(\wdeg C_1:\cdots:\wdeg C_r:\wdeg x_a+\wdeg y_a),
\]
or $\eta=(0:\cdots:0)$.
\end{thm}



\begin{proof}
The case in which $a$ lies on a single component $C_i$ is covered by \cite[Prop.~5.3]{newSolutions}.
If $a$ lies on exactly two components, \cite[Prop.~5.7]{newSolutions} applies under the additional hypothesis
\[
  \wdeg C \;>\; \wdeg x_a + \wdeg y_a .
\]
This extra hypothesis fails only at an $A_1$ singularity; however, by our assumption that $a$ is $\eta$–geometric, this case does not occur.
The general case, where $a$ lies on at least three components, follows by a straightforward adaptation of the argument in \cite[Prop.~5.7]{newSolutions}.
\end{proof}

\newcommand{\pfrac}[2]{( #1 : #2)}
\newcommand{\bigpfrac}[2]{\bigl( #1 : #2\bigr)}
\newcommand{\sfrac}[2]{\frac{#2}{#1}}

\begin{example}
For simple singularities we obtain the values in Table~\ref{tEtaSimple} when all analytic branches lie on distinct components.
If several branches lie on the same component, we add the corresponding entries (i.e.\ sum the branch weights in that coordinate).
For example, three smooth components meeting in a triple point yield $\eta=(1\!:\!1\!:\!1\!:\!2)$, whereas a single component with a triple point yields $\eta=(3\!:\!2)$.

\begin{table} 
\begin{tabular}{|c|c|c|c|c|c}
\hline
Type & Equation & $\eta$, if not (0:0) \\
\hline
node ($A_1$)        & $x^2-y^2$ & 1:1  \\
cusp  ($A_2$)        & $x^2-y^3$ & 6:5  \\
tacnode ($A_3$)    & $(x-y^2)(x+y^2)$ & 2:2:3 \\
$A_n$, $n$ even    & $x^2-y^{n+1}$ & (2n+2):(n+3)  \\
$A_n$, $n$ odd    & $(x-y^{\frac{n+1}{2}})(x+y^{\frac{n+1}{2}})$ & (n+1):(n+1):(n+3)  \\
triple point ($D_4$) & $y(x-y)(x+y)$ & 1:1:1:2  \\
$D_n$, $n$ even    & $y(x^2-y^{\frac{n}{2}-1})(x^2+y^{\frac{n}{2}-1})$ & 2:(n-2):(n-2):n \\
$D_n$, $n$ odd    & $y(x^2-y^{n-2})$ & 2:(2n-4):n \\
$E_6$            & $x^3-y^4$ & 12:7 \\
$E_7$            & $x(x^2-y^3)$ & 9:5   \\
$E_8$            & $x^3-y^5$ & 15:8  \\
\hline
\end{tabular}

\caption{$\eta$ for simple singularities}
\label{tEtaSimple}
\end{table}

\end{example}


\begin{example}
We now demonstrate how $\eta$ can be used in an concrete example to prove Darboux integrability using only the arrangement of singularities of an integral curve:

Let $C$ be the dual curve of a smooth cubic $E$. Since $E$ has $9$ inflection points, $C$ will be a sextic with $9$ cusps and hence total Tjuringa number $18$. Every line through $2$ cusps will be a bitangent and we can assume that one of these is the line at infinity. We are therefore in the situation of Example \ref{ePlaneSextics} and obtain a degree $3$ differential form $\omega$ which has $C$ as integral curve with some cofactor $K$.

At each of the $7$ cusps $a_1,\dots, a_7$ that do not lie at infinity we have $\eta = (6:5)$ or $(0:0)$.  This implies that $5K(a_i)-6d\omega(a_i) = 0$ for $i =1\dots 7$. Now either $5K-6d\omega$ vanishes identically or it defines a conic in the plane that passes through the $7$ cusps. The second case is not possible since then the intersection number of the conic and the sextic would be $2\cdot 7 = 14 > 2\cdot6$. It follows that $d\omega = \frac{5}{6} K$. But this is exactly the condition for the existence of a Darboux integrating factor.

Notice the following: If we start with an elliptic curve $E$ defined over the real numbers (or even the rationals), we obtain $C, \omega, d\omega $ and $K$ over the same field, but the coordinates of the cusps are in general defined only over $\CC$. Nevertheless the argument above still implies $d\omega = \frac{5}{6}K$ and hence $\omega$ still has a Darboux integrating factor. In other words, even singularities which are not defined over the real numbers, do help to proving the existence a Darboux integrating factor for a real differential form.
\end{example}

\section{computing the invariant at infinity}

There are two ways to define $\eta$ for a point $a$ at infinity: One is to use Definition \ref{dEta} locally at a chart around $a$ with respect to coordinates in that chart. The other is to consider $d\omega, C_i$ and $K_i$ in Definition \ref{dEta} as homogeneous polynomials and evaluate at the projective coordinates of $a$. Since $d\omega$ and the $K_i$ are of the same degree $\deg \omega -1$ and their ratio is well defined. While the first definition is easier to predict from the singularities of $C$, it is the second definition that fits best with our application to Darboux integrability.

Unfortunately the values one obtains from these two definition are not equal.
In this section we show how one can obtain one from the other.

We start with some notation: Consider the projective plane $\PP^2$ with coordinates $(X\colon Y \colon Z)$,
the chart $U_{Z\not=0}$ with coordinates $x := \frac{X}{Z}$ and $y:=\frac{Y}{Z}$, and the chart
$U_{X\not=0}$ with coordinates $y :=  \frac{Y}{X}$ and $z := \frac{Z}{X}$. The transformation map $\varphi$ from $U_{X\not=0}$ to $U_{Z\not=0}$, defined on the intersection, is then given by $\varphi(y,z) = \left(\frac{1}{z},\frac{y}{z}\right)$. We call $U_{X \not=0}$ the {\sl chart at infinity}.

\begin{defn} \label{dHomogenization} \label{dPrime}
Let $F$ be a polynomial differential $i$-form, $i=0,1,2$ on $U_{Z\not=0}$ with coefficients  of degree $d$. Then we denote its homogenization by $F^h$ i.e
\[
    F^h := F\left(\frac{X}{Z},\frac{Y}{Z}\right) \cdot \left\{
    \begin{matrix}
    Z^d &\text{if $F$ is a polynomial} \\
    Z^{d+2} &\text{if $F$ is a $1$-form} \\
    Z^{d+3} & \text{if $F$ is a $2$-from}.
    \end{matrix}
    \right.
\]
Dehomogenizing this with respect to $X$ wie obtain a polynomial differential $i$-form on the chart $U_{X\not=0}$ at infinity we denote this by
\[
    F' := F^h(1\colon y \colon z).
\]
Notice that this implies
\[
F' = \varphi^*(F) \cdot \left\{
    \begin{matrix}
    z^d &\text{if $F$ is a polynomial} \\
    z^{d+2} &\text{if $F$ is a $1$-form} \\
    z^{d+3} & \text{if $F$ is a $2$-from}.
    \end{matrix}
    \right.
\]
where $\varphi^*$ is the pull-back map associated to $\varphi$.
\end{defn}

\begin{example}
We have
\[
    (\dx)' = \varphi^*(\dx) z^2 = d\left(\frac{1}{z}\right) z^2 = -\frac{1}{z^2}\dz \cdot z^2 = -\dz
\]
similarly
\[
    (\dy)' = \varphi^*(\dy) z^2 = d\left(\frac{y}{z}\right) z^2 = \frac{z\dy-y\dz}{z^2} \cdot z^2 =     z\dy-y\dz
\]
\end{example}

We can now formally define the extention of our invariant to infinity:

\begin{defn} 
Let $\omega$ be a differential form with algebraic integral curve $C$ and cofactor $K_C$ and $a\in \PP^2$ a projective point. Then
\[
    \eta := \eta(\omega,C,a) := \bigpfrac{(K_C)^h(a)}{(d\omega)^h(a)}.
\]
Notice that for $a = (x:y:1)$ a projective point {\sl not} at infinity, this definition agrees with our previous one.
\end{defn}

\begin{rem}
$\eta$ is well defined. Indeed, since $C$ is an integral curve with cofactor $K$, we
have formula
\[
    dC \wedge \omega = CK
\]
and hence $\deg K = \deg \omega -1 = \deg d\omega$.

\end{rem}

\begin{rem}
Let $\omega$ be a differential form with algebraic integral curve $C$ and cofactor $K_C$ and $a=(1:y:z)$ a point in the chart at infinity. On the one hand, we have by definition
\[
    \eta (\omega,C,a) = \bigpfrac{(K_C)'(a)}{(d\omega)'(a)},
\]
while computing in the chart at infinity, would give
\[
    \eta(\omega',C',a) = \bigpfrac{K_{C'}(a)}{d(\omega')(a)},
\]
assuming that $C'$ is an algebraic integral curve of $\omega'$  with some cofactor $K_{C'}$ (which we prove below).
Unfortunately these values differ, since taking differentials does not commute with homogenization.
\end{rem}

Let us look in more detail at the geometry in the chart at infinity:

\begin{prop} \label{pLineAtInfinity}
Let $\omega$ be a differential form. Then the line at infinity defined by $z=0$ is an algebraic integral curve of $\omega'$ with cofactor
$$
    K_z := \frac{\dz \wedge \omega'}{z}.
$$
In particular this quotient is a polynomial 2-form.
\end{prop}

\begin{proof}
Let $\omega = P \dx + Q \dy$. Then
\begin{align*}
    \dz \wedge \omega'
    &= \dz \wedge \bigl(P' (\dx)' + Q' (\dy)'\bigr) \\
    &= \dz \wedge \bigl(-P' \dz + Q' (z\dy -y\dz)\bigr) \\
    &= Q'z \dz \wedge \dy
\end{align*}
So indeed $\omega' \wedge \dz$ is divisible by $z$. We now check, that $z$ is an integral curve of $\omega'$ with cofactor $K_z$
\[
     \dz \wedge \omega'
    = z \frac{\dz \wedge \omega'}{z}
    = z K_z
\]
\end{proof}

\begin{prop} \label{pRelationPrime}
Let $\omega$ be a differential form with algebraic integral curve $C$  and cofactor $K_C$, then $C'$ is an algebraic integral curve of $\omega'$.

 If we denote the cofactor of $C'$ by $K_{C'}$, we have
the following formulas:
\begin{align*}
    (K_C)' &= K_{C'}  - (\deg C) K_z  \\
     (d\omega)' &= d(\omega') - (\deg \omega +2) K_z.
\end{align*}
\end{prop}

\begin{proof}
Since $C$ is an integral curve of $\omega$ with cofactor $C$ and $\varphi^*$ is a ring homomorphism with have
\[
     \varphi^*(dC) \wedge \varphi^*(\omega) = \varphi^*(C) \cdot \varphi^*(K_C). \quad \quad (\ast)
\]
Multiplying both sides with $z^{\deg C  + \deg \omega + 3}$ we obtain
\[
    (dC)' \wedge \omega' = zC'\cdot (K_C)' .
\]
Since takting differentials commutes with pullback, we have
\begin{align*}
    d(C')
    &= d\bigl(z^{\deg C} \, \varphi^*(C)\bigr)  \\
    &= (\deg C)z^{\deg C-1} \, \varphi^*(C)\dz + z^{\deg C} \varphi^*(dC) \\
    &= \frac{(\deg C) C' \dz}{z}+ \frac{(dC)'}{z}
\end{align*}
and hence
\begin{align*}
    d(C') \wedge \omega'
    &= \frac{(\deg C) C' \dz}{z} \wedge \omega'  + \frac{(dC)'}{z} \wedge \omega' \\
    &= (\deg C) C' \frac{\dz \wedge \omega'}{z} + \frac{zC' \cdot (K_C)'}{z} \quad \quad \text{(using $(\ast)$)}  \\
    &= (\deg C) C' \cdot K_z + C' \cdot (K_C)' \\
    &= C' \cdot \bigl((\deg C) K_x + (K_C)'\bigr).
\end{align*}
If follows that $C'$ is an integral curve of $\omega'$ with cofactor
\[
    K_{C'} = (\deg C) K_z + (K_C)'.
\]
Solving for $(K_C)'$ gives the first formula.

We also have
\begin{align*}
    d(\omega')
    &= d\bigl(z^{\deg \omega+2} \,\varphi^*(\omega)\bigr) \\
    & = (\deg \omega + 2) z^{\deg \omega+1} dz \wedge \varphi^*(\omega) + z^{\deg \omega+2}\varphi^*(d\omega)\\
    & = (\deg \omega + 2) \frac{\dz \wedge \omega'}{z} + (d\omega)' \\
    & = (\deg \omega + 2)  K_z + (d\omega)'.
\end{align*}
which gives the second formula.
\end{proof}

With this data we define a second invariant at a point at infinity, which can be computed geometrically in the chart at infinity using the results of Section \ref{sAnalytic}.

\begin{defn} \label{dEtaPrime}
Let $\omega$ be a differential form with algebraic integral curve $C$ and and $a \in U_{X\not=0}$ a point in the chart at infinity. Then we define
\[
    \eta' := \eta'(\omega,C,a) := \eta\bigl(\omega',\{z,C'\},a\bigr).
\]
This is well defined, since $C'$ and $z$ are algebraic integral curves of $\omega'$.
\end{defn}

We get the following relatation between $\eta$ and $\eta'$:

\begin{prop} \label{pEtaInfinity}
Let $\omega$ be a differential form  with algebraic integral curve $C$. Let $a \in U_{X\not=0}$ be a point in the chart at infinity such that $\omega(a) = 0$. If
\[
    \eta'(\omega,C,a) = (k_z:k:w)
\]
then
\[
    \eta(\omega,C,a) = \bigl(k-(\deg C) k_z:w-(\deg \omega+2) k_z\bigr).
\]
In other words, $\eta\in\PP^1$ is the linear projection of $\eta'\in\PP^2$ from the point
$$(1:\deg C:\deg\omega+2)$$ onto the line
$\{k_z=0\}\subset\PP^2$, identified with $\PP^1$ via $(0:k:w)\mapsto (k:w)$.
\end{prop}

\begin{proof}
By definition we have
\[
     (k_z:k:w) = \bigr(K_z(a):K_{C'}(a):d\omega'(a) \bigl)
\]
Using Proposition \ref{pRelationPrime} we can compute
\begin{align*}
    \eta(\omega,C,a)
    &= \bigl(K_C^h(a) : (d\omega)^h(a) \bigr) \\
    &= \bigl(K_C'(a) : (d\omega)'(a) \bigr) \\
    &= \bigl(K_{C'}(a) - (\deg C)K_z(a) : d\omega'(a) - (\deg \omega +2) K_z(a) \bigr) \\
    &= \bigl(k - (\deg C)k_z : w - (\deg \omega +2) k_z \bigr).
\end{align*}
\end{proof}

\begin{rem}
If $a\in U_{X\neq 0}$ lies \emph{not} on the line at infinity $L_\infty=\{z=0\}$, then $k_z(a)=0$. Hence
\[
  \eta'(\omega,C,a)=(0:k:w),
\]
and the projection formula of Proposition~\ref{pEtaInfinity} reduces to
\[
  \eta(\omega,C,a)=(k-(\deg C)\cdot 0:\,w-(\deg\omega+2)\cdot 0)=(k:w).
\]
In other words, away from $L_\infty$ we may identify $\eta$ with the last two coordinates of $\eta'$.
Geometrically this just says that in this case $\eta'$ already lies on the line $k_z=0$ and is
therefore unchanged by the projection.
\end{rem}

\begin{rem}[Extension to several integral curves] \label{rProjectionSeveralCurves}
Proposition~\ref{pEtaInfinity} extends verbatim to a collection of integral curves $\{C_1,\dots,C_r\}$. In particular, $\eta(\omega,\{C_i\},a)\in\PP^{r}$ is the projection of
$\eta'(\omega,\{C_i\},a)\in\PP^{r+1}$  from the point
\[
(1:\deg C_1:\cdots:\deg C_r:\deg\omega+2)\in\PP^{r+1}.
\]
to the line $k_z=0$.  The proof is identical to the case $r=1$.
\end{rem}

\begin{rem}
Combining Proposition~\ref{pEtaInfinity} with Theorem~\ref{tEtaQuasiHom}, we see that for quasi-homogeneous singularities of $C\cup\{z=0\}$ one can compute $\eta'$ from the local geometry, and recover $\eta$ by incorporating the global information encoded in the degrees of the integral curves.\end{rem}

\begin{example}
Let $\omega$ be a differential form with integral curve $C$ which is triply tangent to the line at infinity at a point $a$. In this situation $C \cap \{z=0\}$ has an $A_3$ singularity at $a$ and a local equation is $z(z-y^3)=0$.
This equation is quasi homogeneous with respect to $\wdeg z=3, \wdeg y=1$. It follows that
\[
    \eta'
    = (\deg_a z : \deg_a C :\deg_a y_a + \deg_a z_a)
    = (3:3:4).
\]
Projecting this form $(1 :\deg C : \deg \omega + 2)$ gives
\[
    \eta = \bigl(3-3(\deg C) : 4 - 3(\deg \omega + 2) \bigr).
\]
Notice that only the degree of $C$ and $\omega$ are used in this formula, but neither the
equation of $\omega$ nor the cofactor $K_C$ of $C$. If for example $\deg C = 3$ and $\deg \omega = 1$
we obtain
\[
    \eta = (-6 : -5) = (6 : 5).
\]
Let us compare this to a direct computation for an explicit example:
The differential form $\omega =  3x\dy - 2y\dx$ is has an integral curve
$C = x^2-y^3$ with cofactor $K_C = 6 \dx \wedge \dy$ since
\[
    dC \wedge \omega
    = (2x\dx - 3y^2\dy) \wedge (3x\dy - 2y\dx)
    = 6(x^2-y^3)\dx \wedge \dy.
\]
We have $C^h = X^2Z - Y^3$ and $C' = z-y^3$ and hence $C'$ has $\{z=0\}$ as a triple tangent. So the geometry of $\omega$ and $C$ agrees with the above assumptions.

Now $d\omega = 3\dx\wedge \dy - 2 \dy \wedge \dx = 5 \dx \wedge \dy$ and we can compute
\[
    \eta = \bigpfrac{K_C(a)}{d\omega(a)} = (6:5).
\]
This agrees with the value predicted by the geometry.
\end{example}

We now arrive at the main result of the paper: a concrete criterion for Darboux integrability that is computable purely from the geometry of $C\cup L_\infty$, without any further knowledge of $\omega$.
The main improvement over \cite{newSolutions} is that the present form also incorporates quasi-homogeneous singularities \emph{on} the line at infinity (via the projection step for $\eta'$), and thus applies in substantially more cases.

\begin{thm}\label{tEtaCriterion}
Let $C=C_1\cup\cdots\cup C_r$ be an algebraic integral curve of the differential form $\omega$, and let
$P_1,\dots,P_k$ be the $\eta$-geometric points of $C\cup L_\infty$. Set
\[
  M_{\eta'} \;=\;
  \begin{pmatrix}
    K_z(P_1) & K_{C_1}(P_1) & \cdots & K_{C_r}(P_1) & d\omega(P_1) \\
    \vdots   & \vdots       &        & \vdots       & \vdots \\
    K_z(P_k) & K_{C_1}(P_k) & \cdots & K_{C_r}(P_k) & d\omega(P_k) \\
    1        & \deg C_1     & \cdots & \deg C_r     & \deg\omega+2
  \end{pmatrix},
\]
the matrix whose $j$-th row is $\eta_j' := \eta'(\omega,\{C_1,\dots,C_r\},P_j)$ and whose last row is the projection center from Remark~\ref{rProjectionSeveralCurves}. Note that, by Theorem~\ref{tEtaQuasiHom}, all entries of $M_{\eta'}$ are determined by the geometry of $C\cup L_\infty$ alone.

If $\mathrm{rank}\,M_{\eta'}\le r+1$ and the points $P_1,\dots,P_k$ do not lie on any curve of degree $\deg\omega-1$,
then $\omega$ is Darboux integrable.
\end{thm}

\begin{proof}
By the rank condition there exists a hyperplane in $\PP^{r+1}$ that contains all points $\eta_j'$ \emph{and} the projection center. Because the hyperplane contains the center, its image under the linear projection
$\pi \colon \PP^{r+1}\dashrightarrow\PP^{r}$ is again a hyperplane; hence the projected points
$\pi(\eta'_j)$ all lie on a hyperplane in $\PP^{r}$. 

By Remark~\ref{rProjectionSeveralCurves} we have
$\pi(\eta'_j) = \eta(\omega,\{C_1,\dots,C_r\},P_j)$.
It follows, that there exist coefficients $(\lambda_1,\dots,\lambda_r,\lambda_{r+1})\neq 0$ such that
\[
  \sum_{i=1}^r \lambda_i\,K^h_{C_i}(P_j)\;+\;\lambda_{r+1}\,d\omega^h(P_j)\;=\;0
  \qquad (j=1,\dots,k).
\]
Define
\[
  Q \;:=\; \sum_{i=1}^r \lambda_i\,K^h_{C_i}\;+\;\lambda_{r+1}\,d\omega^h .
\]
Then $Q$ is a homogeneous polynomial of degree $\deg\omega-1$ and $Q(P_j)=0$ for all $j$. By hypothesis, no nonzero polynomial of degree $\deg\omega-1$ vanishes on all $P_j$, hence $Q\equiv 0$. This yields a nontrivial linear relation among the cofactors and $d\omega$,
so $\omega$ is Darboux integrable.
\end{proof}

\section{Zeros outside of integral curves}
\nosubsections

\newcommand{\tC}{{\widetilde C}}

Consider a differential form $\omega$ with an integral curve $C$. A necessary condition for the existence of a center of $\omega$ at a point $a \in \CC^2$ is $\omega(a) = d\omega(a) = 0$. Notice that for quasi-homogeneous singularities $a$ of $C$ we have $\omega(a)=0$ but usually $d\omega(a) \not=0$, see Theorem \ref{tEtaQuasiHom}. In this section we will concern ourselves with the question where and under what conditions points with $\omega(a) = d\omega(a)=0$ exist.

\begin{prop}
Consider a differential form $\omega$ with an integral curve $C$ and cofactor $K_C$. If $a \not\in C$ is a zero of $\omega$  then $K_C(a) = 0$.
\end{prop}

\begin{proof}
Since $C$ is an integral curve of $\omega$ we have
\[
 dC \wedge \omega = C \cdot K_C
\]
where $K_C$ is the cofactor of $C$. Since $\omega(a) =0$ but $C(a) \not=0$ we must have $K_C(a) = 0$
\end{proof}

We now use Darboux' condition for the existence of an integrating factor:

\begin{prop}
Let $\omega$ be a differential form, $C_1,\dots,C_r$ algebraic integral
curves of $\omega$, $K_1,\dots,K_r$ their cofactors and $a$ a zero of $\omega$ with $a \not\in C_i$ for $i=1\dots r$. Assume that
\[
    \sum \alpha_i K_i = -d\omega
\]
for appropriate $\alpha_i \in \CC$. Then $d\omega(a) = 0$.
\end{prop}

\begin{proof}
By the previous Proposition we have $K_i(a) = 0$ for $i=1\dots r$ and therefore
\[
    d\omega(a) = -\sum \alpha_i K_i(a) = 0.
\]
\end{proof}

We want to estimate the number of zeros of $\omega$ that lie outside of  an integral curve $C$ using the only the geometry of $C$. For this we need a standard Chern class computation:

\begin{lem} \label{lChern}
Let
\[
0 \longrightarrow K \longrightarrow E
\longrightarrow
L \longrightarrow Q \longrightarrow 0
\]
be an exact sequence of coherent sheaves on $\PP^2$, where
\begin{enumerate}
  \item $E$ is a vector bundle of rank $r$,
  \item $L$ is a line bundle,
  \item $Q$ is zero‐dimensional of length
  $\ell =\deg(Q)$.
\end{enumerate}
Then $K$ is a vector bundle of rank $r-1$, and its Chern classes satisfy
\[
\begin{aligned}
  c_{1}(K)  &=  c_{1}(E)-c_{1}(L), \\
  c_{2}(K)  &=  c_{2}(E)-c_{1}(E)\,c_{1}(L)+c_{1}(L)^{2}-\ell.
\end{aligned}
\]
\end{lem}

\begin{proof}
Since $E$ and $L$ are locally free, $K$ is a second syzygy sheaf of $Q$.
As $Q$ has finite support on the smooth surface $\PP^2$, any second syzygy
is locally free: see \cite[II, §1, Def.~1.1.5 \& Thm.~1.1.6]{OSS}.
Thus $K$ is a vector bundle of rank $r-1$.

Exactness gives the alternating Whitney identity for total Chern classes
\[
  c(K)\,c(L)=c(E)\,c(Q).
\]
Here $c(L)=1+c_1(L)$ since $L$ is a line bundle, and since $Q$ is zero-dimensional of length $\ell$, we have $c(Q)=1-\ell\,[\mathrm{pt}]$.
Comparing degrees $1$ and $2$ yields
\[
\begin{aligned}
  c_1(K)+c_1(L)&=c_1(E),\\
  c_2(K)+c_1(K)c_1(L)&=c_2(E)-\ell,
\end{aligned}
\]
and hence
\[
\begin{aligned}
  c_1(K) & =c_1(E)-c_1(L) \\
  c_2(K) &= c_2(E)-c_1(E)c_1(L)+c_1(L)^2-\ell
\end{aligned}
\]
\end{proof}

With this we can compute our formula:

\begin{prop} \label{pNumZeros}
Let $C$ be an algebraic integral curve of a differential form $\omega$ and $X$ the finite scheme defined by $C_x=C_y=C=0$. Then $\omega$ has at most
\[
    (\deg \omega)^2 +
    (\deg C)(\deg C - \deg \omega - 1) - \deg X
\]
zeros outside of $C$.
\end{prop}

\begin{proof}
Let $a$ be a zero of $\omega = P\dx + Q\dy$ outside of $C$, i.e $P(a) = Q(a) = 0$ and $C(a) \not=0$. Since $C$ is an integral curve of $\omega$ we have
\[
    dC \wedge \omega = CK_C
\]
where $K_C$ the cofactor of $C$. Since $\omega(a)=0$ and $C(a) \not=0$ we must also have $K_C(a) = 0$.

Hence every zero $a\notin C$ of $\omega$ lies in the zero scheme
\[
Y:=V(P,Q,K_C).
\]
In particular, the number of zeros of $\omega$ outside $C$ (without multiplicity) is bounded above by $\deg Y$; with multiplicities.

Consider the exact sequence
\[
  0 \rightarrow \sK \rightarrow
  \underbrace{\sO(s)\oplus\sO(s)\oplus\sO(s-1)}_{=:E}
  \xrightarrow{(C_x,\;C_y,\;C)}
  \underbrace{\sO(s+d-1)}_{=:L}
  \rightarrow \sO_X \rightarrow 0,
\]
with \(s=\deg\omega\) and \(d=\deg C\).
From
\[
dC\wedge\omega = CK_C
\quad\Longleftrightarrow\quad
C_xQ-C_yP-CK_C = 0
\]
we see that $(Q,\,-P,\,-K_C)^t\colon \sO\to E$ factors through \(\sK\). Hence it defines a section of $\sK$
whose zero scheme is $Y$.  By Lemma \ref{lChern} $\sK$ is a vector bundle and hence, if $Y$ is zero dimensional,  $\deg Y = c_2(\sK)$.

To apply the formula of Lemma \ref{lChern} we compute the
chern classes of $E$ and $L$ in our current situation:
\begin{align*}
    c_1(L) &= s+d-1  \\
    c_1(E) & = 3s-1 \\
    c_2(E) & = s^2 + 2s(s-1) = 3s^2 - 2s.
\end{align*}
We therefore have
\begin{align*}
    c_2(K)
    &= c_{2}(E)-c_{1}(E)c_{1}(L)+c_{1}(L)^{2}-\ell \\
    &= (3s^2-2s) - (3s-1)(s+d-1)+(s+d-1)^2 - \deg X \\
    &= s^2 - sd + d^2 - d - \deg X \\
    &= s^2 + d(d - s - 1) - \deg X.
\end{align*}
\end{proof}

\begin{example} \label{eOnePointOnly}
If $C$ is a sextic with $\deg X = 20$ then there exists an $\omega$ of degree $3$ with $C$ as integral curve. The proposition above gives that there is at most
\[
    3^2 + 6(6-3-1) - 20 = 1
\]
zero of $\omega$ outside of $C$.
\end{example}

We now count the number of zeros of $\omega$ at infinity.

\begin{prop} \label{pZerosAtInfinity}
Let $C$ be an integral curve of a differential form $\omega$ and $X_\infty$ the finite scheme defined by $C_x=C_y=C=z=0$ over some field $k$.  Assume furthermore that
\[
        \deg X_\infty \le \deg C - \deg \omega - 2
\]
and that the characteristic of $k$ does not divide $\deg C$.
Then $\omega$ has, counted with multiplicity, exactly
\[
    \deg \omega - 1
\]
zeros at infinity.
\end{prop}

\begin{proof}
Since $C$ is an integral curve of $\omega$, we have
\[
  (C_x,C_y,C)\cdot(Q,-P,-K_C)^{t}=0,
\]
where $K_C$ is the cofactor of $C$.
Denote by the subscript $\infty$ the restriction to the line at infinity $\{z=0\}$.
Then
\[
  (Q_\infty,-P_\infty,-K_{C,\infty})^{t}
\]
is a syzygy of $(C_{\infty,x},C_{\infty,y},C_\infty)$ over $k[x,y]$.

We claim that the columns of
\[
\begin{pmatrix}
x &\frac{1}{g}C_{\infty,y}\\
y &-\frac{1}{g}C_{\infty,x} \\
-\deg C& 0
\end{pmatrix}
\qquad\text{with } \; g:=\gcd(C_{\infty,x},C_{\infty,y})
\]
generate the $k[x,y]$-module $\mathrm{Syz}(C_{\infty,x},C_{\infty,y},C_\infty)$.
Indeed, the first column comes from the  Euler relation
\[
  xC_{\infty,x}+yC_{\infty,y} = (\deg C)C_\infty.
\]
Since $\deg C\neq 0$ over $k$, we may assume that all further generators of the syzygy module have zero in the third entry. All such syzygies are between $C_{\infty,x}$ and $C_{\infty,y}$ only and are therefore generated by the second column.

Thus $(Q_\infty,-P_\infty,-K_{C,\infty})^{t}$ is a $k[x,y]$-linear combination of these two columns.
Note that $\deg g=\deg X_\infty$, so the degree of the second column equals
\[
  \deg(C)-1-\deg X_\infty\ge \deg \omega+1
\]
by our assumption on $X_\infty$.
Therefore the combination for
$$(Q_\infty,-P_\infty,-K_{C,\infty})^{t}$$
must involve only the first column.
The coefficient then has degree $\deg \omega -1$, and hence $\omega$ has $\deg \omega-1$ zeros at infinity (counted with multiplicity).
\end{proof}

\begin{example} \label{eMaximalInfinity}
Let $C\subset\PP^2$ be a maximal sextic, i.e.\ a sextic with simple singularities and maximal Milnor number $19$. Assume in addition that the line at infinity meets $C$ with multiplicity $2$ and transversely elsewhere. In this situation the modified Tjurina number of $C$ is $20$, and differential forms $\omega$ of degree $3$ with integral curve $C$ exist. By Example~\ref{eOnePointOnly} there is at most one zero of $\omega$ outside $C$.

Now $\deg X_\infty=1$ and $\deg C-\deg\omega-2=1$; therefore, by Proposition~\ref{pZerosAtInfinity}, we have at least
\[
  \deg\omega-1=2
\]
zeros on the line at infinity. Generically we expect at most one of them at the multiplicity-$2$ point and none at the transverse intersections. Hence, generically, the second point lies outside $C$ but on the line at infinity. It follows that, even if we can prove that $\omega$ admits a Darboux integrating factor with respect to $C$, we expect no centers of $\omega$ outside $C$ and the line at infinity. For this reason we do not consider maximal sextics in our constructions.
\end{example}

\begin{example}  \label{eSubMaximalZeros}
Let $C\subset\PP^2$ be a sub-maximal sextic, i.e.\ a sextic with simple singularities and Milnor number $18$. Assume in addition that the line at infinity meets $C$ with multiplicities $(2,2,1,1)$ or $(3,1,1,1)$. In this situation the modified Tjurina number of $C$ is again $20$, and differential $1$-forms $\omega$ of degree $3$ with integral curve $C$ exist.
Again, we expect only one zero of $\omega$ outside $C$, \emph{but} here $\deg X_\infty=2$ and $\deg C-\deg\omega-2=1$, so Proposition~\ref{pZerosAtInfinity} does not apply. Indeed, in examples the expected zero lies away from the line at infinity. All constructions in this article are of this type.
\end{example}

\section{Review of results in finite characteristic}
\nosubsections

In this section we recall the results obtained by our group in \cite{survey}, \cite{ste:2011}.

Let $X_{13} \subset (\FF_{29})^{14}$ denote the common vanishing locus of the first thirteen focal values in the space of normalized degree-3 differential forms. We work in characteristic $29$, which is the smallest prime for which Frommer's algorithm is well defined for the first thirteen focal values.

In \cite{centerfocusWeb} Jakob Kr\"oker developed a very fast C++ implementation of Frommer's algorithm and tested, for (almost) all points of $(\FF_{29})^{14}$, whether they lie on $X_{13}$. He also computed the codimension of the Zariski tangent space of $X_{13}$ at each point found. This required about 11 CPU-years on a compute cluster, and the results are collected in a database available at \cite{centerFocusDatabase}.

Using only the number of points found one can heuristically estimate the number of irreducible components of $X_{13}$ in each codimension. Indeed a variety of dimension $d$ has on the order of $p^{d}$ points over $\FF_{p}$ (a very rough heuristic inspired by the Weil conjectures, see \cite{weilconjectures},\cite{deligneproof}). Using this heuristic, Kr\"oker~\cite{survey} obtains the estimates collected in the second column of Table~\ref{tab:steiner-vs-kroeker}. In that computation, points are bucketed by the codimension $r$ of the Zariski tangent space. At singular points of $X_{13}$ the codimension of the tangent space  drops below codimension of the containing component, so such points are assigned to lower–codimension buckets. This depresses the counts in higher codimensions; thus the second column should be read as a conservative lower bound.

\begin{table}
\centering
\begin{tabular}{c|c|c|c}
\hline
\textbf{codimension $r$} & \textbf{by counting points} & \textbf{by interpolation} & \textbf{kown}\\
\hline
5  & $\ge 1.00\pm 0.00$ & 1 & 1 \\
6  & $\ge 1.96\pm 0.00$ & 2 & 2\\
7  & $\ge 3.84\pm 0.00$ & 4 & 4\\
8  & $\ge 3.83\pm 0.00$ & 4 & 4\\
9  & $\ge 11.97\pm 0.04$    & 14 & 10 \\ 
10 & $\ge 32.85\pm 0.37$    & 41 & 4\\
11 & $\ge 77.50\pm 3.07$    & 98 & 4\\
\hline
\end{tabular}
\caption{Estimated numbers of irreducible components of $X_{13}\subset \FF_{29}^{14}$ that contain a smooth $\FF_{29}$-rational point of $X_{13}$.}
\label{tab:steiner-vs-kroeker}
\end{table}

Based on Kr\"oker’s database, Johannes Steiner~\cite{ste:2011} computed a heuristic decomposition of $X_{13}$ over $\FF_{29}$, together with low-degree candidate equations for the conjectured components. The method lifts smooth points of $X_{13}$ to higher-order local solutions (jets) via a Hensel-type procedure and then uses interpolation to recover low-degree polynomials vanishing on these jets (a finite-field analogue of~\cite{SomVerWam:2001}). Counting distinct candidate ideals, he conjectures the component counts in the third column of Table~\ref{tab:steiner-vs-kroeker}. Concrete low-degree generators for the candidate ideals are available at~\cite{cfGitHub}.

\begin{defn}
Let $\omega\in(\FF_{29})^{14}$ be a normalized differential. If $\omega\in X_{13}$, let $r$ be the codimension of the Zariski tangent space of $X_{13}$ at $\omega$.

We say that $\omega$ has \emph{Steiner type} $r_i$ if $\omega\in V_{r,i}\cap X_{13}$, where $V_{r,i}$ is Steiner’s heuristic component as defined in \cite{ste:2011}. Note that a differential form may have several Steiner types, since the $V_{r,i}$ intersect.

If $\omega$ has integer coefficients, we say it has Steiner type $r_i$ if its reduction mod~$29$ can be normalized to a form of Steiner type $r_i$. Again, a differential form may have several Steiner types, since different normalization centers of $\omega$ can yield different types.
\end{defn}

Johannes Steiner also assigned a selection of Jakob Kr\"oker’s solutions
to his heuristic components $V_{r,i}$ and, for each, computed all algebraic integral curves of degree at most $6$ having at least one smooth $\FF_{29}$-rational point. This yields a treasure trove of geometric information about these components, which we have used in \cite{newSolutions} and will also use here.

To our knowledge this is the only study that treats the entire $14$-dimensional parameter space of normalized degree-$3$ differential forms and the full ideal generated by the first thirteen focal values. There are, however, caveats:
\begin{enumerate}
\item It is not proved that the heuristic components $V_{r,i}$ are genuine components of $X_{13}$ in characteristic $29$, although Steiner’s tests provide strong evidence (with one doubtful case).
\item Even if they are components of $X_{13}$ in characteristic $29$, it is unknown whether they are components of $X_{\infty}$ in that characteristic (i.e., whether they admit a center). Many examples are Darboux integrable with low-degree invariant curves, but not all.
\item Even if they are components of $X_{\infty}$ in characteristic~29, persistence to characteristic~0 is unknown. The fourth column of Table~\ref{tab:steiner-vs-kroeker} lists the irreducible components in characteristic~0~known to us (\cite{ZoladekRational}, \cite{ZoladekCorrection}, \cite{survey}, \cite{torregrosa},\cite{newSolutions}); all these components are of Steiner type, but many predicted Steiner types remain unrealized in characteristic~0.
\end{enumerate}

\section{constructions}
\nosubsections

\newcommand{\minihead}[1]{%
  \par\medskip
  \noindent\underline{\emph{#1}}\enspace
}

In this section we use the methods developed in this paper to construct several codimension~$11$ components of the degree~$3$ center variety; most of which are, to our best knowledge, new.
We collect the needed ingredients in the following  proposition.

\begin{prop}[Blueprint for constructing codim-$11$ components] 
\label{pBlueprint}

Assume there exist
\begin{enumerate}
  \item a plane sextic $C=C_1\cup\cdots\cup C_r$ with simple singularities whose Milnor numbers add up to $18$ (i.e.\ $C$ is submaximal),
  \item a line $L_\infty$ bitangent to $C$ (after a projective change of coordinates
        we take $L_\infty$ to be the line at infinity),
  \item a degree-$3$ differential $1$-form $\omega$ having $C$ as an integral curve
        (existence of $\omega$ follows from (1)–(2)),
\end{enumerate}
such that the following additional properties hold.
\begin{enumerate}
\item[(a)] The matrix $M_{\eta'}$ from Theorem \ref{tEtaCriterion}
  has rank at most  $r+1$.
  \item[(b)] The points $\eta$-geometric points of $C\cup L_\infty$ do not lie on a conic.
  \item[(c)] $\omega$ has a zero outside $C\cup L_\infty$.
  \item[(d)] For some prime $p$, after reducing the coefficients of $\omega$ modulo $p$ and normalizing for Frommer’s algorithm, the first $13$ focal values vanish (this is automatic from (a) and (b)) and the associated $13\times 13$ Jacobian
  has rank $11$ over $\FF_p$ (this is not).
\end{enumerate}
Assume finally that
\begin{enumerate}
  \item[(e)] there exists a $1$-dimensional family of not projectively equivalent submaximal sextics with at least the same singularity
  basket as $C$ (this ist automatic for submaximal sextics) and the same component structure $C=C_1\cup\cdots\cup C_r$ (this is not)
\end{enumerate}
Then there is a codimension-$11$ component of the degree-$3$ center variety whose general member
is Darboux-integrable, with an invariant sextic of the type described in (1)–(2).
\end{prop}

\newcommand{\sWirr}{\sW_{irr}}
\newcommand{\sCirr}{\sC_{irr}}
\newcommand{\sWnorm}{\sW_{norm}}
\newcommand{\sCnorm}{\sC_{norm}}
\newcommand{\sVnorm}{\sV_{norm}}
\newcommand{\sDnorm}{\sD_{norm}}
\newcommand{\sHnorm}{\sH_{norm}}

\begin{proof}
Let $\sC$ be the equisingular locus of plane sextics satisfying (1)–(2), and put $\sH$ for the vector space of degree-3 polynomial $1$-forms $P\dx + Q\dy$. Set
\[
  \sW := \{(\omega',C') \in \sH\times\sC \mid \text{$C'$ is an integral curve of $\omega'$}\}.
\]
Let $\pi_1:\sW\to\sH$ and $\pi_2:\sW\to\sC$ be the projections to the first and second factor, respectively. For $C'\in\sC$, the fiber of $\pi_2$ is the affine vector space
\[
  \pi_2^{-1}(C')=\sV_{C'}(3)=\{\omega'\in\sH \mid \text{$C'$ is an integral curve of $\omega'$}\}.
\]
By Theorem~\ref{pExpected} we have
\[
\dim \sV_{C'}(3) \ge \binom{3-6+3}{2} + \binom{3+1}{2} - (6-1)^2 + \deg X \;=\; \deg X - 19.
\]
Since $C'$ has total Milnor number $18$ by (1) and is bitangent to $L_\infty$ by (2), we get $\deg X=18+2=20$, hence $\dim \sV_{C'}(3)\ge 1$. In particular, $\pi_2$ is surjective.

By (1)–(3) we have a point $(\omega,C)\in\sW$. Let $\sWirr$ be the irreducible component of $\sW$ containing $(\omega,C)$.

Property (a) holds on $\sWirr$, since the $\eta$–matrix depends only on the local singularity types of $C\cup L_\infty$ and on the fixed component decomposition, which are constant along $\sW$. Property (b) is Zariski open and holds at $(\omega,C)$ by assumption, hence it holds on a nonempty open subset of $\sWirr$.

By Example~\ref{eSubMaximalZeros} the number of zeros of $\omega$ in $\PP^{2}\setminus(C\cup L_\infty)$ is at most $1$ and is lower semicontinuous for forms satisfying (1)–(3). Hence the locus where (c) holds is Zariski open in $\sWirr$; it contains $(\omega,C)$ and is therefore nonempty.

Let $\sWirr^{0}\subset\sWirr$ be the nonempty Zariski-open locus where (b) and (c) hold.
For any $(\omega',C')\in\sWirr^{0}$, condition (a) gives $\operatorname{rank} M_{\eta'}\le r+1$, and (b) says that the $\eta$-geometric points do not lie on a conic, i.e. on a curve of degree $\deg\omega'-1=2$.
Thus the hypotheses of Theorem~\ref{tEtaCriterion} are satisfied, and $\omega'$ is Darboux integrable.

Let $\sHnorm \subset \sH$ be the affine subspace of differential $1$-forms that vanish at the origin and whose linear part is $x\dx+y\dy$. This space has dimension $14$.
Now let $\sWnorm := \sWirr^{0} \cap (\sHnorm \times \sC)$ be the locally closed subvariety of $\sWirr^{0}$ whose differential forms are normalized and $\sCnorm := \pi_2(\sWnorm)$ the sextic curves that occur as integral curves of normalized differential forms in $\sWnorm$.

To see that $\sWnorm$ is nonempty let $(\omega_0,C_0)$ be the pair obtained from $(\omega,C)$ by translating the zero from (c) to the origin. All properties of the proposition are invariant under translation, so $(\omega_0,C_0)\in\sWirr^{0}$. Since $\omega_0$ is Darboux integrable and the origin is not on the integral curve $C_0$, the $2\times 2$ matrix $S$ of linear coefficients of $\omega_0$ is symmetric. Hence there exists $A\in\mathrm{GL}_2$ with $A^{\mathsf T}SA=I$, and after the linear change of variables $(x,y)\mapsto A^{-1}(x,y)$ we obtain $(\omega_{\mathrm{norm}},C_{\mathrm{norm}})\in\sWnorm$.

It follows that $\pi_1(\sWnorm)$ is a family of normalized Darboux–integrable degree-3 differential $1$-forms with a sextic integral curve satisfying (1)–(3) and (a)–(c).
Property (d) then implies that $\pi_1(\sWnorm)$ lies in a component $Y$ of the center variety having codimension at least $11$ in the $14$-dimensional affine space of all normalized degree-$3$ differential forms $\sHnorm$, i.e. it has dimension at most $3$.

Let $\sD \subset \sC$ be the $\mathrm{PGL}_3$–saturation of the family from (e), i.e. the union of all projective images of its members.
On $\sD$ the condition “has exactly the singularities prescribed in (1)–(2)” is Zariski open; since $C \in \sD$ satisfies (1)–(2), this open set is nonempty.
Set $\sDnorm := \sD \cap \sCnorm$.
Restricting the fibration $\sWnorm \to \sCnorm$ over $\sDnorm$ yields a subfamily
\[
  \sVnorm \longrightarrow \sDnorm.
\]
Then $\pi_1(\sVnorm)$ is also a family of normalized Darboux–integrable degree-3 differential $1$-forms with a sextic integral curve satisfying (1)–(3) and (a)–(c). We now estimate $\dim \pi_1(\sVnorm)$.

Recall that the space of normalized degree-3 forms $\sHnorm$ is invariant under rotations about the origin and under scalings
\[
  \omega_{\mathrm{norm}}(x,y)\ \longmapsto\ \frac{1}{\lambda^{2}}\,\omega_{\mathrm{norm}}(\lambda x,\lambda y).
\]
The only fixed form under this $2$-dimensional group is $x\dx+y\dy$; every other form has a $2$-dimensional orbit, and forms on the same orbit have projectively equivalent integral curves. By normalizing the differential forms of the curves given in (e) we obtain a $1$-dimensional family of forms in $\pi_1(\sVnorm)$ whose integral curves are not projectively equivalent. The saturation of this family by rotations and scalings still lies inside $\pi_1(\sVnorm)$ and has dimension at least $1+2=3$. Since the ambient space of normalized forms $\sHnorm$  has dimension $14$, it follows that
\[
  \codim \pi_1(\sVnorm) \le\ 14-3\ =\ 11.
\]
From (d) we already know that $\pi_1(\sWnorm) $ is contained in a component $Y$ of the center variety with $\codim Y \ge 11$. Because $\pi_1(\sVnorm) \subset \pi_1(\sWnorm) \subset Y$, we have
\[
  11 \le\ \codim Y \le \codim\pi_1(\sWnorm) \le \codim \pi_1(\sVnorm) \ \le\ 11,
\]
so all three codimensions are equal to $11$. In particular, $Y=\overline{\pi_1(\sWnorm)}=\overline{\pi_1(\sVnorm)}$ is an irreducible codimension-$11$ component of the degree-$3$ center variety whose generic element $(\omega_{\mathrm{norm}},C_{\mathrm{norm}})$ satisfies (1)–(3) and (a)–(e).
 \end{proof}
 
 To apply this proposition one could, in principle, proceed as follows:
\begin{enumerate}
  \item Find all possible singularity baskets and component structures that satisfy (1) and (2). This yields a large but finite list of cases.
  \item Select those with at least six $\eta$-geometric points.
  \item Select those for which $M_{\eta'}$ drops rank.
  \item Construct a $1$-dimensional family of examples.
  \item Check (b) and (c) for a single example.
  \item Reduce modulo a suitable finite field, normalize, and check (d).
\end{enumerate}
The main bottlenecks are step (4), which typically requires substantial hand work, and step (6), which often fails. In that case one has constructed a codimension-$11$ family that is not itself a component but lies inside a larger component. At present we cannot predict from the configuration whether this happens.

In this paper we use a shortcut. Jakob Kr\"oker~\cite{centerFocusDatabase} found many differential forms over the finite field $\FF_{29}$ that satisfy (d). Johannes Steiner \cite{ste:2011} computed the $\FF_{29}$-rational integral curves of degree at most $6$ for many of these examples. We scanned these examples for configurations satisfying (1), (2), and (a); the resulting list is given in Table~\ref{tSteiner}. Only for these configurations did we attempt (and solve) the manual construction step in characteristic $0$.

\begin{table}
\centering
\begin{tabular}{c c c c}
\hline
\textbf{Steiner type}  & \textbf{degrees} & \textbf{\# $\eta$-geom.\ pts} & \textbf{construction}\\
\hline
$11_2$   & $\{2,4\}$     & 6 & \ref{cA2A4}\\
$11_{18}$  & $\{1,1,4\}$   & 6 &  \ref{cA1A4}\\
$11_{25}$  & $\{2,4\}$     & 7 & \ref{c3A2}\\
$11_{27}$  & $\{1,1,4\}$   & 6 &  \ref{cA1A4}\\
$11_{53}$  & $\{1,1,4\}$   & 7 &  \ref{c3A1}\\
$11_{59}$  & $\{1,1,4\}$   & 6 &  \ref{cA1A4}\\
\hline
\end{tabular}

\quad

\caption{Conjectured codimension-$11$ components of the center variety over $\FF_{29}$ that contain examples satisfying (1)–(3), and (a)–(d) of Proposition~\ref{pBlueprint}, each with an invariant sextic whose component degrees are shown.} \label{tSteiner}
\end{table}

We present our constructions as follows:
\begin{enumerate}
  \item \textbf{Configuration.}\label{step:config}
        Propose a configuration of components and singularities that satisfies (1) and (2).
  \item \textbf{Darboux integrability.}\label{step:integrability}
        Compute $M_{\eta'}$ and verify (a).
  \item \textbf{Construction.}\label{step:explicit}
        Construct a $1$-dimensional family of curves (e).
         For one curve $C$ in the family, compute $\omega$; check (b) and (c).
\item \textbf{Reduction modulo $p$.}\label{step:reduction}
      Reduce mod $p$, normalize, and check (d).
\end{enumerate}
The computations are carried out in \texttt{Macaulay2}; the scripts are available online \cite{codim11M2}.

The first two constructions (Steiner types $11_2$ and $11_{25}$) start with a maximal quartic $C_4$ ($\mu=6$) and a conic $C_2$ having prescribed contact multiplicities with $C_4$: $\{2,2,2,2\}$ in the first case and $\{3,3,1,1\}$ in the second.
A contact of multiplicity $m$ between two distinct components contributes $2m-1$ to the Milnor number of the union (this applies when the contact occurs along a single branch at a simple singularity). Hence
\(
\mu =6+4\cdot 3=18
\)
in the first case and
\(
\mu=6+2\cdot 5+2\cdot1=18
\)
in the second. Thus $C_6=C_4\cup C_2$ is submaximal in both cases.

\subsection{\boldmath $A_2\!+\!A_4$-quartic \& contact conic (Steiner type $11_2$)}
\label{cA2A4}

\subsubsection{Configuration.}
Let $C_6=C_4\cup C_2$, where $C_4$ is a quartic and $C_2$ is a smooth conic. Assume the following hold:
\begin{enumerate}[label=(\roman*)]
  \item $C_4$ has an $A_2$ singularity at a point $A$ and an $A_4$ singularity at a point $B$.
  \item $C_2$ is tangent to $C_4$ at four distinct points, one of them being $B$. Denote the others by $U,V,W$.
  \item The line at infinity $L_\infty$ passes through $A$ with multiplicity $2$ and is tangent to $C_4$ at a further smooth point $T$.
\end{enumerate}
If no further singularities occur, then $C_6$ is submaximal: its singularities are
 $A_2$ at $A$, $D_7$ at $B$ and $A_3$ at the smooth contact points $U,V,W$
Hence the total milnor number is $2+3+3+3+7=18$. Such a configuration would satisfy $(1)$ and $(2)$ of Proposition \ref{pBlueprint}. 

\subsubsection{Darboux integrability.}
The $\eta$-geometric points of $C_6\cup L_\infty$ are $A,B,T,U,V,W$.
Looking up the values of $\eta'$ for the corresponding simple singularities in Table \ref{tEtaSimple}, we obtain
\[
M_{\eta'} =
\begin{pmatrix}
 0 & 10 & 2 & 7\\  
 2 &  6 & 0 & 5\\  
 2 &  2 & 0 & 3\\  
 0 &  2 & 2 & 3\\  
 0 &  2 & 2 & 3\\  
 0 &  2 & 2 & 3\\  
 1 &  4 & 2 & 5    
\end{pmatrix}
\quad
\begin{matrix}
\text{($D_7$ at $A$)}\\
\text{($D_5$ at $B$)}\\
\text{($A_3$ at $T$)}\\
\text{($A_3$ at $U$)}\\
\text{($A_3$ at $V$)}\\
\text{($A_3$ at $W$)}\\
\text{(the projection center)}
\end{matrix}
\]
where the columns are $(K_z,\,K_{C_4},\,K_{C_2},\,d\omega)$. We put $0$ in the column of any component that does not pass through the given point.

We observe that $M_{\eta'}\cdot(2,1,2,-2)^{\mathsf t} = 0$, so the rank of $M$ is at most $3$ and condition (a) is satisfied. Hence each $\omega$ having this configuration of integral curves would be Darboux-integrable.

\subsubsection{Explicit construction.}
It remains to realize the configuration described above and show that it lies in a family as in (e). Classically, any plane quartic carrying an $A_2$ and an $A_4$ singularity is projectively equivalent to
\[
  C_4 \colon  (xy-z^2)^2-xz^3 = 0,
\]
and in these coordinates the $A_2$ lies at $A=(1\!:\!0\!:\!0)$ while the $A_4$ lies at $B=(0\!:\!1\!:\!0)$; see \cite[page 25]{Hui1979}. A direct computation shows that there is a unique line through the $B$ which is tangent to $C_4$ at a smooth point; it is
\[
  L \colon 4y+z=0,
\]
with tangency point $T=(16\!:\!-1\!:\!4)$.

The space of conics in $\PP^2$ is $\PP^5$, and each imposed tangency to a given curve contributes a single condition. Hence the family of contact conics for $C_4$ is at least $1$-dimensional. It remains to exhibit one example with exactly the tangencies described above. For this we recall the following classical construction for contact curves. Let
\[
    M = \begin{pmatrix} F & G \\ G & H \end{pmatrix}
\]
a homogeneous symmetric matrix. Then $F=0$ is contact to $\det M=0$. Indeed
on $F = 0$ we have $\det M = G^2$ and therefore all intersection points are of multiplicity $2$.

Here we look at
\[
    M =  \left(\!\begin{array}{cc}
      -x^{2}-4xy-2xz-4yz+3z^{2}&2xy+xz+4yz-z^{2}\\
      2xy+xz+4yz-z^{2}&-4y\,z-z^{2}
      \end{array}\!\right)
\]
Since $\det M = -4C_4$ we have that
\[
    C_2 \colon  -x^{2}-4xy-2xz-4yz+3z^{2} = 0
\]
is contact to $C_4$. On also checks that $C_2$ is smooth, that the four contact points are distinct, and that one of them is $B$. Furthermore one can check that the six \(\eta\)-geometric points are not contained in a conic. This proves that a curve as in (1), (2) exist and that is satisfies (a), (b) and (e).

We now construct $\omega$ as in (3) and check (c):
Applying the substitution $z\mapsto z-4y$ sends the bitangent line $L$ to the line at infinity $z=0$, and using computer algebra one finds a unique (up to scaling) degree-$3$ differential $1$-form $\omega$ for which both $C_4$ and $C_2$ are integral curves. This $\omega$ has a zero at $(71\!:\!10\!:\!51)$; after translating so that this zero is at the origin, an affine expression is:

\begin{align*}
    \omega = (&867x^{2}y-81498xy^{2}-194208y^{3} \\
    &+170x^{2}+6868xy-64702y^{2} \\
    &+145x+3450y)dx \\
      + (&1734x^{3}+112710x^{2}y-138720xy^{2}-2663424y^{3} \\
      &+5066x^{2}+372130xy+209984y^{2}\\
      &+3450x+279380y) dy
    \end{align*}
Its linear part is symmetric, hence \(\omega\) has a center at \((0,0)\).

\subsubsection{Reduction modulo $p$.}
Over the finite field $\FF_{29}$ this differential form can be normalized to
\begin{align*}
&(x^{3}+4x^{2}y+3xy^{2}-2y^{3}-4x^{2}+14xy+6y
      ^{2}+x)\dx \\
      &+(-14x^{3}+5x^{2}y-2y^{3}-3x^{2}+11xy+3y^{2}+y
)\dy.
\end{align*}

With Frommer's Algorithm we can compute (as a sanity check) that the first $13$ focal values vanish, and that the Jacobian matrix of the first $13$ focal values has rank $11$. This shows (d) and we can apply Proposition \ref{pBlueprint} to obtain a codim 11 component of the center variety in degree $3$.

\subsection{\boldmath $3A_2$-quartic \& $3,3$-contact conic (Steiner type $11_{25}$)}
\label{c3A2}

\subsubsection{Configuration.}
Let $C_6=C_4\cup C_2$, where $C_4$ is a quartic and $C_2$ is a smooth conic. Assume the following hold:
\begin{enumerate}[label=(\roman*)]
  \item $C_4$ has three $A_2$ singularities at points $U,V,W$.
  \item $C_2$ has contact of order $3$ with $C_4$ at two distinct points $A,B$, and meets $C_4$ transversely at two further points $R,S$.
  \item The line at infinity $L_\infty$ passes through $S$ with multiplicity $2$ and is tangent to $C_4$ at a further smooth point $T$.
  \item The points $A,B,R,S,T,U,V,W$ do not lie on a conic.
\end{enumerate}
If no further singularities occur, then $C_6$ is submaximal: its singularities are
$A_2$ at $U,V,W$, $A_5$ at $A,B$, and $A_1$ at $R,S$.
Hence the total Milnor number is $2+2+2+5+5+1+1=18$. Such a configuration satisfies (1) and (2) of Proposition~\ref{pBlueprint}.

\subsubsection{Darboux integrability.}
The $\eta$-geometric points are $U,V,W,A,B,S,T$. Looking up the quasi-homogeneous
values of $\eta'$ at these points (in this order) yields
\[
\begin{pmatrix}
 0 &  6 & 0 & 5\\  
 0 &  6 & 0 & 5\\  
 0 &  6 & 0 & 5\\  
 0 &  3 & 3 & 4\\  
 0 &  3 & 3 & 4\\  
 1 &  1 & 1 & 2\\  
 2 &  2 & 0 & 3\\  
 1 &  4 & 2 & 5    
\end{pmatrix}
\quad
\begin{matrix}
\text{($A_2$ at $U$)}\\
\text{($A_2$ at $V$)}\\
\text{($A_2$ at $W$)}\\
\text{($A_5$ at $A$)}\\
\text{($A_5$ at $B$)}\\
\text{($D_4$ at $S$)}\\
\text{($A_3$ at $T$)}\\
\text{(projection center)}
\end{matrix}
\]
A quick check yields $M_\eta \cdot (4, 5,3,-6)^t = 0$ and hence $\rank M_\eta \le 3$. This proves $(a)$.
    
\subsubsection{Explicit construction.}
It remains to realize the configuration described above and to show that it lies in a family as in (e).
Classically, any plane quartic with three $A_2$ singularities is projectively equivalent (after a projective change of coordinates) to
\[
  C_4:\ x^2y^2 + y^2z^2 + z^2x^2 - 2xyz(x + y + z)=0,
\]
and in these coordinates the $A_2$ singularities $U,V,W$ lie at the coordinate points; see \cite[page 14]{Hui1979}.
A direct computation shows that
\[
   C_2:\ 9x^2 - 80xy - 432y^2 - 80xz + 1880yz - 432z^2 = 0
\]
is a smooth conic meeting $C_4$ with intersection multiplicity $3$ at
$A=(36\!:\!9\!:\!4)$ and $B=(36\!:\!4\!:\!9)$, and transversely at
$R=(4\!:\!9\!:\!36)$ and $S=(4\!:\!36\!:\!9)$.
Furthermore,
\[
  L:\ 27x + 125y - 512z = 0
\]
meets $C_4$ and $C_2$ transversely at $S$ and is tangent to $C_4$ at
the point $T=(1600\!:\!-576\!:\!-225)$.
One checks that the points $U,V,W,A,B,R,S,T$ do not lie on a conic.
Thus these curves realize the configuration described above.

To see that this configuration lies in a family as in (e), we do a dimension count.
The space of conics in $\PP^2$ has dimension $5$, and the existence of two points of
multiplicity-$3$ contact with $C_4$ imposes $2\cdot 2=4$ conditions.
The two remaining simple intersections impose no further condition. Hence $C_2$
lies in a $1$-dimensional family with at least the prescribed contacts.
Furthermore one can check that the six $\eta$-geometric points are not contained in a conic. This proves that a curve as in (1), (2) exist and that it satisfies (a), (b) and (e).

We now construct $\omega$ as in (3) and check (c):
Applying the substitution $z\mapsto z+\frac{27x+125y}{512}$ sends the bitangent line $L$ to the line at infinity $z=0$, and using computer algebra one finds a unique (up to scaling) degree-$3$ differential $1$-form $\omega$ for which both $C_4$ and $C_2$ are integral curves. This $\omega$ has a zero at $(256\!:\!-256\!:\!273)$; after translating so that this zero is at the origin, an affine expression is:
\begin{align*}
    \omega = (& -38233377x^{3}+1624359555x^{2}y-7790988387xy^{2}+
      8443514625y^{3}\\
      &-2731921920x^{2}+33127564288xy+
      1632816640y^{2}-33512488960x \\
      &-17154703360y)dx \\
      - (&439348455x^{3}-1706788629x^{2}y-462700875xy^{2}+
      916085625y^{3}\\
      &+9614212608x^{2}-34486000640xy+
      4299742720y^{2}+17154703360x-\\
      &27556577280y) dy
    \end{align*}
Its linear part is symmetric, hence \(\omega\) has a center at \((0,0)\).

\subsubsection{Reduction modulo $p$.}
Over the finite field $\FF_{29}$ this can be normalized to
\begin{align*}
&(-8x^{3}-13x^{2}y+5xy^{2}-11y^{3}-10x^{2}-2xy+14y^{2}+x)\dx \\
      &+(14x^{3}+10x^{2}y+12xy^{2}-9y^{3}-2x^{2}+14xy-14y^{2}+y)\dy.
\end{align*}

With Frommer's Algorithm we can compute  that the first $13$ focal values vanish, and that the Jacobian matrix of the first $13$ focal values has rank $11$. This shows (d) and we can apply Proposition \ref{pBlueprint} to obtain a codim 11 component of the center variety in degree $3$.

\subsection{\boldmath $A_1{+}A_4$-quartic \& two bitangents (Steiner types~$11_{18}$, $11_{27}$, $11_{59}$)} \label{eThreeBitangents}
\label{cA1A4}

 \subsubsection{Configuration.}
Let $C_4$ be a quartic, and let $L,L',L''$ be (generalized) bitangents—i.e., lines meeting $C_4$ with
intersection multiplicity $2$ at two distinct points.  Assume
\begin{enumerate}[label=(\roman*)]
  \item $C_4$ has an $A_1$ singularity at $A$ and an $A_4$ singularity at $B$;
  \item $L$ intersects $C_4$ in two smooth points $S$ and $T$;
  \item $L'$ intersects $C_4$ in $A$ and a further smooth point $U$
  \item $L''$ intersects $C_4$ in $B$ and a further smooth point $V$
\end{enumerate}

We now have three possibilities for the choice of the line at infinity:
\begin{enumerate}
\item $L_\infty=L$ and $C_6=C_4\cup L'\cup L''$;
\item $L_\infty=L'$ and $C_6=C_4\cup L\cup L''$;
\item $L_\infty=L''$ and $C_6=C_4\cup L\cup L'$.
\end{enumerate}
In all three cases the configuration satisfies (1)–(2) of Proposition~\ref{pBlueprint}.

\subsubsection{Darboux integrability.}
In all three cases $C_6\cup L_\infty=C_4\cup L\cup L'\cup L''$, and the
$\eta$-geometric points of $C_6\cup L_\infty$ are $A,B,S,T,U,V$.
Looking up the values of $\eta'$ for the corresponding simple singularities in Table~\ref{tEtaSimple}, we obtain:
\[
M_{\eta'} =
\begin{pmatrix}
 0 & 1 & 0 &  2 & 2 \\  
 0 & 0 & 2 & 10 & 7 \\  
 2 & 0 & 0 &  2 & 3 \\  
 2 & 0 & 0 &  2 & 3 \\  
 0 & 2 & 0 &  2 & 3 \\  
 0 & 0 & 2 &  2 & 3 \\  
 1 & 1 & 1 &  4 & 5     
\end{pmatrix}
\quad
\begin{matrix}
\text{($D_4$ at $A$)}\\
\text{($D_7$ at $B$)}\\
\text{($A_3$ at $S$)}\\
\text{($A_3$ at $T$)}\\
\text{($A_3$ at $U$)}\\
\text{($A_3$ at $V$)}\\
\text{(projection center)}
\end{matrix}
\]
where the columns are $(K_L,K_{L'},K_{L''},K_{C_4},\,d\omega)$.
Note that $$M_\eta\cdot(2,2,2,1,-2)^{\mathsf T}=0,$$ hence $\operatorname{rank} M_{\eta'}\le 4$.
Moreover, choosing $L$, $L'$, or $L''$ as the line at infinity merely permutes the first three columns,
so the rank of $M_{\eta'}$ is unchanged. This proves (a) in all three cases.

\subsubsection{Explicit construction.}
It remains to realize the configuration described above and to show that it lies in a family as in (e).
Classically, any plane quartic with an $A_1$ and an $A_4$ singularity is projectively equivalent to
\[
  C_4:\ (xz+y^2)^2 + xy^3 - \lambda\,x^2yz = 0, \qquad \lambda\not= 0,-1,
\]
with the $A_1$ at $A=(1\!:\!0\!:\!0)$ and the $A_4$ at $B=(0\!:\!0\!:\!1)$; see \cite[page 25]{Hui1979} after a projective change of coordinates.
A direct computation shows that, for $\lambda\neq -1,0$, the following lines are (generalized) bitangents satisfying (ii)–(iv):
\begin{align*}
  L&:\ \lambda^{2}x + 4(\lambda-1)y - 16z = 0,\\
  L'&:\ y + 4(\lambda+1)z = 0,\\
  L''&:\ \lambda^{2}x - 4(\lambda+1)y = 0.
\end{align*}
Thus the configuration occurs in a $1$-parameter family as in (e).
From now on we fix $\lambda=1$ and compute the differential forms as in (3) for the three choices of $L_\infty$.

Sending $L$ to infinity via the linear change $z\mapsto \tfrac{x-z}{16}$ we obtain
\begin{align*}
    \omega = (&-3x^{3}+26x^{2}y-64xy^{2}-416y^{3}+6x^{2}+14xy-16y^{2}-3x-16y)\dx \\
      + (&16x^{3}+16x^{2}y-256xy^{2}-768y^{3}-56x^{2}-16xy-128y^{2}+40x) \dy,
\end{align*}
with a zero (off the invariant curves) at $(28\!:\!-9\!:\!40)$.

Sending $L'$ to infinity via $z\mapsto -\tfrac{y-z}{8}$ we obtain
\begin{align*}
    \omega' = (&-83xy^{2}-136y^{3}+38xy+56y^{2}-3x-16y)\dx \\
      + (&6x^{3}+63x^{2}y-56xy^{2}-256y^{3}-35x^{2}-88xy+64y^{2}+40x) \dy,
\end{align*}
with a zero at $(63\!:\!-18\!:\!260)$.

Finally, sending $L''$ to infinity via $(y,z)\mapsto\bigl(\tfrac{x-z}{8},\,y\bigr)$ we obtain
\begin{align*}
    \omega'' = (&-15x^{3}-1664x^{2}y+34816xy^{2}-131072y^{3}\\
    &+33x^{2}+1088xy-2048y^{2}-21x-192y+3
)\dx \\
      + (&1264x^{3}-4096x^{2}y-196608xy^{2}-1808x^{2}\\
      &+5120xy+592x-1024y-48) \dy,
\end{align*}
with a zero at $(664\!:\!-189\!:\!5850)$.

\subsubsection{Reduction modulo $p$.}
Over the finite field $\FF_{29}$ all three differential forms can be normalized for Frommer’s algorithm.
In each case the first $13$ focal values vanish, and the Jacobian matrix of the first $13$ focal values has rank $11$.
This verifies (d), so Proposition~\ref{pBlueprint} yields three codimension-$11$ components of the degree-$3$ center variety.
We also find that $\omega$, $\omega'$, and $\omega''$ have Steiner type $11_{58}$, $11_{27}$, and $11_{18}$, respectively.

\subsection{\boldmath $3A_1$–quartic with a flex line and a hyperflex line (Steiner type $11_{53}$)}
\label{c3A1}

\subsubsection{Configuration.}

\newcommand{\Lflex}{L_{\mathrm{flex}}}
\newcommand{\Lhyper}{L_{\mathrm{hyper}}}

Let $C_6=C_4\cup \Lflex \cup \Lhyper$, where $C_4$ is a quartic and $\Lflex,\Lhyper$ are lines. Assume:
\begin{enumerate}[label=(\roman*)]
  \item $C_4$ has three $A_1$ singularities at points $U,V,W$.
  \item $C_4$ has a flex (inflection) point at $A$, and $\Lflex$ is its flex line; that is,
        $\Lflex$ meets $C_4$ with multiplicity $3$ at $A$ and multiplicity $1$ at a further point $B$.
  \item $C_4$ has a hyperflex at $B$, and $\Lhyper$ is the hyperflex line at $B$; that is,
        $\Lhyper$ meets $C_4$ with multiplicity $4$ at $B$.
  \item $C_4$ admits a bitangent $L_\infty$ with (smooth) contact points $S$ and $T$.
\end{enumerate}
If no further singularities occur, then $C_6$ is submaximal: its singularities are
$A_1$ at $U,V,W$, $A_5$ at $A$, and $D_{10}$ at $B$, so the total Milnor number is
$3+5+10=18$. Hence the configuration satisfies \emph{(1)} and \emph{(2)} of Proposition~\ref{pBlueprint}.

\subsubsection{Darboux integrability.}
The $\eta$-geometric points are $U,V,W,A,B,S,T$. Looking up the quasi-homogeneous
values of $\eta'$ at these points (in this order) yields
\[
\begin{pmatrix}
 0 & 0 & 0 &  2 & 2 \\  
 0 & 0 & 0 &  2 & 2 \\  
 0 & 0 & 0 &  2 & 2 \\  
 0 & 6 & 0 &  6 & 8 \\  
 0 & 2 & 8 &  8 &10 \\  
 2 & 0 & 0 &  2 & 3 \\  
 2 & 0 & 0 &  2 & 3 \\  
 1 & 1 & 1 &  4 & 5     
\end{pmatrix}
\quad
\begin{matrix}
\text{($A_1$ at $U$)}\\
\text{($A_1$ at $V$)}\\
\text{($A_1$ at $W$)}\\
\text{($A_5$ at $A$)}\\
\text{($D_{10}$ at $B$)}\\
\text{($A_3$ at $S$)}\\
\text{($A_3$ at $T$)}\\
\text{(projection center)}
\end{matrix}
\]
with column order $(K_z,\,K_{\Lflex},\,K_{\Lhyper},\,K_{C_4},\,d\omega)$. Zeros indicate components not passing through the given point.
Moreover,
\[
M_{\eta'}\cdot (3,\,2,\,1,\,6,\,-6)^{\mathsf T}=0,
\]
hence $\operatorname{rank} M_{\eta'}\le 4$. This proves $(a)$.

\subsubsection{Explicit construction.}
It remains to realize the configuration above and to exhibit a $1$–parameter family as in (e).
Consider the morphism
\[
  \varphi:\ \PP^1 \longrightarrow \PP^2
\]
defined by
\[
    (s:t) \mapsto \bigl(s^{4}:s t^{3} \colon (s-t)^{2}(s+\lambda t)^{2} \bigr), \qquad
     \lambda\in\CC\setminus\{-1,0,1\}.
\]
Let $C_4$ be the image of $\varphi$, and set
\[
  A=\varphi(1\!:\!0),\quad B=\varphi(0\!:\!1),\quad S=\varphi(1\!:\!1),\quad T=\varphi(\lambda\!:\!-1).
\]
Then:
\begin{enumerate}[label=(\roman*)]
\item $C_4$ is rational; by the genus formula a general member is expected to have three $A_1$–singularities.
\item $\Lflex=\{y=0\}$ is a flex line at $A$ and meets $C_4$ again (transversely) at $B$.
\item $\Lhyper=\{x=0\}$ is a hyperflex line at $B$.
\item $L_\infty=\{z=0\}$ is bitangent to $C_4$ at $S$ and $T$.
\end{enumerate}
Thus we obtain a $1$–parameter family as in (e). For $\lambda=4$ one checks that
$C_6=C_4\cup\Lflex\cup\Lhyper$ has exactly the singularities listed in the configuration,
and that the $\eta$-geometric points $U,V,W,A,B,S,T$ do not lie on a conic, establishing (b).

Using computer algebra, one finds a unique (up to scaling) degree-3 differential $1$-form
\begin{align*}
    \omega
    = (& -91x^{2}y-819xy^{2}-4160y^{3}-13xy-240y^{2}-4y) \dx \\
        + (&52x^{3}-1638x^{2}y+6656xy^{2}-68x^{2}+672xy+16x) \dy
\end{align*}
for which $C_6$ is an integral curve, as in (3). The form $\omega$ has a zero at $(28\!:\!-12\!:\!325)$ lying outside $C_6\cup L_\infty$, which establishes (c).

\subsubsection{Reduction modulo $p$.}
Over the finite field $\FF_{29}$, after reduction and normalization at the zero found above, $\omega$ becomes
\begin{align*}
&\bigl(-9x^{2}y+5xy^{2}-4y^{3}-10x^{2}+10xy+2y^{2}+x\bigr)\,dx \\
+&\bigl(x^{3}+3x^{2}y+9xy^{2}-3y^{3}-6x^{2}-10xy+10y^{2}+y\bigr)\,dy.
\end{align*}

Using Frommer’s algorithm, the first $13$ focal values vanish, and the Jacobian matrix of these $13$ focal values has rank $11$.
This verifies (d), and by Proposition~\ref{pBlueprint} we obtain a codimension-$11$ component of the degree-$3$ center variety.

\subsection{\boldmath Cuspidal cubic and a 3-tangent star  (Steiner types $11_{50}$, $11_{64}$)}
\label{cCuspStar}

As a final construction, we show that two of the components in \cite{torregrosa} are recovered by our blueprint.
Noticing that these example fits seamlessly into the $\eta$ framework was a key motivation for the present systematic study of submaximal sextics with a bitangent.

\subsubsection{Configuration.}

\newcommand{\Ccusp}{C_{\mathrm{cusp}}}
\newcommand{\Cstar}{C_{\mathrm{star}}}

Let $C_6=\Ccusp\cup\Cstar$ be the union of two plane cubics. Assume:
\begin{enumerate}[label=(\roman*)]
  \item $\Ccusp$ has a cusp at $A$.
  \item $\Cstar$ has a $D_4$–singularity at $B$, i.e.\ $\Cstar=L_1\cup L_2\cup L_3$ is the union of three distinct lines concurrent at $B$.
  \item For $i=1,2,3$, the line $L_i$ meets $\Ccusp$ with intersection multiplicity $2$ at a smooth point $P_i$
        and transversely at a second smooth point $Q_i$.
  \item The line at infinity $L_\infty$ passes through two of the tangency points, say $Q_1$ and $Q_2$, and is transverse to both branches there.
\end{enumerate}
If no further singularities occur, then $C_6$ is submaximal: its singularities are
$A_2$ at $A$, $D_4$ at $B$, $A_3$ at each $P_i$, and $A_1$ at each $Q_i$, hence the total Milnor number is
\[
\mu(C_6)=2+4+3\cdot 3+3\cdot 1=18.
\]
In addition, $L_\infty$ is a generalized bitangent to $C_6$ so the configuration satisfies (1)–(2) of Proposition~\ref{pBlueprint}.

\subsubsection{Darboux integrability.}
The $\eta$-geometric points are $A$, $B$, $P_1$, $P_2$, $P_3$, $Q_1$, $Q_2$. Looking up the quasi-homogeneous
values of $\eta'$ at these points (in this order) yields
\[
\begin{pmatrix}
 0 & 0 & 6 & 5  \\  
 0 & 3 & 0 & 2  \\  
 0 & 2 & 2 & 3  \\  
 1 & 1 & 1 & 2  \\  
 1 & 3 & 3 & 5     
\end{pmatrix}
\quad
\begin{matrix}
\text{($A_2$ at $A$)}\\
\text{($D_4$ at $B$)}\\
\text{($A_3$ at $P_1,P_2,P_3$)}\\
\text{($D_4$ at $Q_1,Q_2$)}\\
\text{(projection center)}
\end{matrix}
\]
with column order $(K_z,\,K_{\Cstar},\,K_{\Ccusp}, \,d\omega)$. Moreover,
\[
M_{\eta'}\cdot (3,4,5,-6)^{\mathsf T}=0,
\]
hence $\operatorname{rank} M_{\eta'}\le 4$. This proves $(a)$.

\subsubsection{Explicit construction.}
These configurations occur in a natural family. Fix a cuspidal cubic $\Ccusp$, and let $B$ be a general point.
Projection from $B$ induces a degree-$3$ map
$$\pi:\widetilde{\Ccusp}\to\PP^1;$$
since $\widetilde{\Ccusp}\cong\PP^1$, the Riemann–Hurwitz formula gives $\deg R_\pi=4$.
Equivalently, there are exactly four lines through $B$ meeting $\Ccusp$ with multiplicity two: one is the line through the cusp $A$, and the other three are simple tangents at smooth points $P_1,P_2,P_3$.
Let $L_1,L_2,L_3$ denote these tangents, and write $Q_i$ for the third (transverse) intersection of $L_i$ with $\Ccusp$.
Set $\Cstar:=L_1\cup L_2\cup L_3$, and choose $L_\infty$ to be the line through two of the third intersection points, say $Q_1$ and $Q_2$ (for a general choice, $L_\infty$ is transverse to the branches there).

As $B$ varies, this yields a $2$–dimensional family of such configurations; modulo projectivities preserving $\Ccusp$ (a $1$–dimensional group), the projective equivalence classes form a $1$–dimensional family.
This proves (e).
Conditions (b) and (c) hold for a general choice and can be checked explicitly for a sample member, exactly as in the previous examples.

\subsubsection{Reduction modulo $p$.}
A subtle field-of-definition issue appears here. Over an algebraically closed field every differential
form with symmetric linear part can be normalized (i.e.\ brought to linear part $x\,dx+y\,dy$).
Over nonclosed fields such as $\RR$ or $\FF_{29}$ this need not be possible, because it amounts to
diagonalizing a symmetric $2\times2$ matrix over the base field.

In the previous examples we circumvented this by choosing a member of the family whose linear part
\emph{is} diagonalizable over the base field. In the present family, however, diagonalizability over
$\FF_{29}$ occurs precisely when the tangent lines $L_1$ and $L_2$, are \emph{conjugate over}
$\FF_{29}$ (each defined only over a quadratic extension, while their union $L_1\cup L_2$ is defined
over $\FF_{29}$). The same phenomenon occurs over $\RR$ (cf.\ \cite{torregrosa}), where one only obtains a real
equation for $L_1\cup L_2$. (The parallel behavior over \(\RR\) and \(\FF_{29}\) is not structural; it need not persist over other finite fields.)

Respecting this constraint, we nevertheless find examples over $\FF_{29}$ that can be normalized and
satisfy condition (d); these are of Steiner type~$11_{64}$. Because in this type examples only $L_3$ is
rational over $\FF_{29}$, only this lines appears in the Kr\"oker–Steiner database, and for the
same reason this case is absent from Table~\ref{tSteiner}.

\subsubsection{Switching bitangents.}
As in Example~\ref{eThreeBitangents}, we may also take one of the generalized bitangents $L_1,L_2,L_3$ as the line at infinity.
Over $\CC$ this yields two further families: one obtained from $L_1$ (equivalently from $L_2$, since $L_1$ and $L_2$ are conjugate) and one from $L_3$.
Over $\RR$ and over $\FF_{29}$ the first family does not occur, because the line at infinity must be defined over the base field. The second does occur and gives a family of Steiner type $11_{50}$.

\def\cprime{$'$} \def\cprime{$'$}

\end{document}